\documentclass{amsart}
\usepackage{amssymb,url,color}
\usepackage{tikz}
\usepackage{adjustbox}
\usetikzlibrary{arrows,arrows.meta,calc,matrix,shapes,cd,patterns,decorations.pathreplacing}

\newenvironment{pic}[1][]
{\begin{aligned}\begin{tikzpicture}[#1]}
{\end{tikzpicture}\end{aligned}}

\newenvironment{minipic}[1][]
{\begin{aligned}\begin{tikzpicture}[font=\tiny,baseline={(0,-.1)},#1]\tikzstyle{morphism}=[minimorphism]\tikzstyle{dot}=[minidot]}
{\end{tikzpicture}\end{aligned}}

\def\strarr{length=2pt, width=3pt}
\tikzset{arrow/.style={decoration={
    markings,
    mark=at position #1 with \arrow{>[\strarr]}},
    postaction=decorate},
    reverse arrow/.style={decoration={
    markings,
    mark=at position #1 with {{\arrow{<[\strarr]}}}},
    postaction=decorate}
}

\tikzstyle{dot}=[circle, draw=black, fill=white, inner sep=.4ex, on layer=foreground]
\tikzstyle{minidot}=[circle, draw=black, fill=white, inner sep=.3ex, on layer=foreground]
\tikzstyle{blackdot}=[dot, fill=black!50]
\tikzstyle{whitedot}=[dot, fill=white]

\newif\ifvflip\pgfkeys{/tikz/vflip/.is if=vflip}
\newif\ifhflip\pgfkeys{/tikz/hflip/.is if=hflip}
\newif\ifhvflip\pgfkeys{/tikz/hvflip/.is if=hvflip}
\newlength\morphismheight
\newlength\minimorphismheight
\newlength\wedgewidth
\tikzset{width/.initial=1mm}
\makeatletter
\tikzstyle{morphism}=[font=\small,morphismshape]
\tikzstyle{minimorphism}=[font=\tiny,minimorphismshape]
\pgfdeclareshape{morphismshape}
{
    \savedanchor\centerpoint
    {
        \pgf@x=0pt
        \pgf@y=0pt
    }
    \anchor{center}{\centerpoint}
    \anchorborder{\centerpoint}
    \saveddimen\overallwidth
    {
        \pgfkeysgetvalue{/tikz/width}{\minwidth}
        \pgf@x=\wd\pgfnodeparttextbox
        \ifdim\pgf@x<\minwidth
            \pgf@x=\minwidth
        \fi
    }
    \savedanchor{\upperrightcorner}
    {
        \pgf@y=.5\ht\pgfnodeparttextbox
        \advance\pgf@y by -.5\dp\pgfnodeparttextbox
        \pgf@x=.5\wd\pgfnodeparttextbox
    }
    \anchor{north}
    {
        \pgf@x=0pt
        \pgf@y=0.5\morphismheight
    }
    \anchor{north east}
    {
        \pgf@x=\overallwidth
        \multiply \pgf@x by 2
        \divide \pgf@x by 5
        \pgf@y=0.5\morphismheight
    }
    \anchor{east}
    {
        \pgf@x=\overallwidth
        \divide \pgf@x by 2
        \advance \pgf@x by 5pt
        \pgf@y=0pt
    }
    \anchor{west}
    {
        \pgf@x=-\overallwidth
        \divide \pgf@x by 2
        \advance \pgf@x by -5pt
        \pgf@y=0pt
    }
    \anchor{north west}
    {
        \pgf@x=-\overallwidth
        \multiply \pgf@x by 2
        \divide \pgf@x by 5
        \pgf@y=0.5\morphismheight
    }
    \anchor{south east}
    {
        \pgf@x=\overallwidth
        \multiply \pgf@x by 2
        \divide \pgf@x by 5
        \pgf@y=-0.5\morphismheight
    }
    \anchor{south west}
    {
        \pgf@x=-\overallwidth
        \multiply \pgf@x by 2
        \divide \pgf@x by 5
        \pgf@y=-0.5\morphismheight
    }
    \anchor{south}
    {
        \pgf@x=0pt
        \pgf@y=-0.5\morphismheight
    }
    \anchor{text}
    {
        \upperrightcorner
        \pgf@x=-\pgf@x
        \pgf@y=-\pgf@y
    }
    \backgroundpath
    {
    \begin{scope}
        \pgfkeysgetvalue{/tikz/fill}{\morphismfill}
        \pgfsetstrokecolor{black}
        \pgfsetlinewidth{.7pt}
        \begin{scope}
        \pgfsetstrokecolor{black}
        \pgfsetfillcolor{white}
        \pgfsetlinewidth{.7pt}
                \ifhflip
                    \pgftransformyscale{-1}
                \fi
                \ifvflip
                    \pgftransformxscale{-1}
                \fi
                \ifhvflip
                    \pgftransformxscale{-1}
                    \pgftransformyscale{-1}
                \fi
                \pgfpathmoveto{\pgfpoint
                    {-0.5*\overallwidth-5pt}
                    {0.5*\morphismheight}}
                \pgfpathlineto{\pgfpoint
                    {0.5*\overallwidth+5pt}
                    {0.5*\morphismheight}}
                \pgfpathlineto{\pgfpoint
                    {0.5*\overallwidth + \wedgewidth}
                    {-0.5*\morphismheight}}
                \pgfpathlineto{\pgfpoint
                    {-0.5*\overallwidth-5pt}
                    {-0.5*\morphismheight}}
                \pgfpathclose
                \pgfusepath{fill,stroke}
        \end{scope}
    \end{scope}
    }
}
\pgfdeclareshape{minimorphismshape}
{
    \savedanchor\centerpoint
    {
        \pgf@x=0pt
        \pgf@y=0pt
    }
    \anchor{center}{\centerpoint}
    \anchorborder{\centerpoint}
    \saveddimen\overallwidth
    {
        \pgfkeysgetvalue{/tikz/width}{\minwidth}
        \pgf@x=\wd\pgfnodeparttextbox
        \ifdim\pgf@x<\minwidth
            \pgf@x=\minwidth
        \fi
    }
    \savedanchor{\upperrightcorner}
    {
        \pgf@y=.5\ht\pgfnodeparttextbox
        \advance\pgf@y by -.5\dp\pgfnodeparttextbox
        \pgf@x=.5\wd\pgfnodeparttextbox
    }
    \anchor{north}
    {
        \pgf@x=0pt
        \pgf@y=0.5\minimorphismheight
    }
    \anchor{north east}
    {
        \pgf@x=\overallwidth
        \multiply \pgf@x by 2
        \divide \pgf@x by 5
        \pgf@y=0.5\minimorphismheight
    }
    \anchor{east}
    {
        \pgf@x=\overallwidth
        \divide \pgf@x by 2
        \advance \pgf@x by 5pt
        \pgf@y=0pt
    }
    \anchor{west}
    {
        \pgf@x=-\overallwidth
        \divide \pgf@x by 2
        \advance \pgf@x by -5pt
        \pgf@y=0pt
    }
    \anchor{north west}
    {
        \pgf@x=-\overallwidth
        \multiply \pgf@x by 2
        \divide \pgf@x by 5
        \pgf@y=0.5\minimorphismheight
    }
    \anchor{south east}
    {
        \pgf@x=\overallwidth
        \multiply \pgf@x by 2
        \divide \pgf@x by 5
        \pgf@y=-0.5\minimorphismheight
    }
    \anchor{south west}
    {
        \pgf@x=-\overallwidth
        \multiply \pgf@x by 2
        \divide \pgf@x by 5
        \pgf@y=-0.5\minimorphismheight
    }
    \anchor{south}
    {
        \pgf@x=0pt
        \pgf@y=-0.5\minimorphismheight
    }
    \anchor{text}
    {
        \upperrightcorner
        \pgf@x=-\pgf@x
        \pgf@y=-\pgf@y
    }
    \backgroundpath
    {
    \begin{scope}
        \pgfkeysgetvalue{/tikz/fill}{\morphismfill}
        \pgfsetstrokecolor{black}
        \pgfsetlinewidth{.7pt}
        \begin{scope}
        \pgfsetstrokecolor{black}
        \pgfsetfillcolor{white}
        \pgfsetlinewidth{.7pt}
                \ifhflip
                    \pgftransformyscale{-1}
                \fi
                \ifvflip
                    \pgftransformxscale{-1}
                \fi
                \ifhvflip
                    \pgftransformxscale{-1}
                    \pgftransformyscale{-1}
                \fi
                \pgfpathmoveto{\pgfpoint
                    {-0.5*\overallwidth-5pt}
                    {0.5*\minimorphismheight}}
                \pgfpathlineto{\pgfpoint
                    {0.5*\overallwidth+5pt}
                    {0.5*\minimorphismheight}}
                \pgfpathlineto{\pgfpoint
                    {0.5*\overallwidth + \wedgewidth}
                    {-0.5*\minimorphismheight}}
                \pgfpathlineto{\pgfpoint
                    {-0.5*\overallwidth-5pt}
                    {-0.5*\minimorphismheight}}
                \pgfpathclose
                \pgfusepath{fill,stroke}
        \end{scope}
    \end{scope}
    }
}

\pgfkeys{%
  /tikz/on layer/.code={
    \pgfonlayer{#1}\begingroup
    \aftergroup\endpgfonlayer
    \aftergroup\endgroup
  },
  /tikz/node on layer/.code={
    \gdef\node@@on@layer{%
      \setbox\tikz@tempbox=\hbox\bgroup\pgfonlayer{#1}\unhbox\tikz@tempbox\endpgfonlayer\pgfsetlinewidth{\thickness}\egroup}
    \aftergroup\node@on@layer
  },
  /tikz/end node on layer/.code={
    \endpgfonlayer\endgroup\endgroup
  }
}
\def\node@on@layer{\aftergroup\node@@on@layer}
\makeatother

\pgfdeclarelayer{foreground}
\pgfdeclarelayer{background}
\pgfdeclarelayer{morphismlayer}
\pgfdeclarelayer{edgelayer}
\pgfdeclarelayer{nodelayer}
\pgfsetlayers{background,main,morphismlayer,foreground,edgelayer,nodelayer}

\tikzset{halo/.style={
         preaction={draw,white,line width=4pt,-},
         preaction={draw,white,ultra thick, shorten >=-2.5\pgflinewidth}}}

\DeclareMathOperator{\ISub}{ISub}
\DeclareMathOperator{\T}{\mathcal{T}}
\DeclareMathOperator{\dom}{dom}
\DeclareMathOperator{\cod}{cod}
\DeclareMathOperator{\cO}{\mathcal{O}}
\DeclareMathOperator{\downset}{\downarrow}
\renewcommand{\S}{\mathcal{S}}
\newcommand{\cat}[1]{\ensuremath{\mathbf{#1}}}
\newcommand{\op}{\ensuremath{^{\mathrm{op}}}}
\newcommand{\id}[1][]{\ensuremath{1_{#1}}}
\newcommand{\total}[1]{\ensuremath{\lceil #1 \rceil}}
\newcommand{\rest}[1]{\ensuremath{\overline{#1}}}
\newcommand{\range}[1]{\ensuremath{\widehat{#1}}}

\theoremstyle{plain}
\newtheorem{theorem}{Theorem}[section]
\newtheorem{proposition}[theorem]{Proposition}
\newtheorem{corollary}[theorem]{Corollary}
\newtheorem{lemma}[theorem]{Lemma}

\theoremstyle{definition}
\newtheorem{definition}[theorem]{Definition}
\newtheorem{example}[theorem]{Example}
\newtheorem{remark}[theorem]{Remark}

\begin{document}
\title{Tensor-restriction categories}
\author{Chris Heunen}
\address{University of Edinburgh, United Kingdom}
\email{chris.heunen@ed.ac.uk}
\author{Jean-Simon Pacaud Lemay}
\address{University of Oxford, United Kingdom}
\email{jean-simon.lemay@kellogg.ox.ac.uk}
\date{\today}
\thanks{We thank Nick Bezhanishvili, Cole Comfort, Richard Garner, Rory Lucyshyn-Wright, Chad Nester, and Sean Tull for useful discussions.}
\begin{abstract}
 Restriction categories were established to handle maps that are partially defined with respect to composition.
 Tensor topology realises that monoidal categories have an intrinsic notion of space, and deals with objects and maps that are partially defined with respect to this spatial structure.
 We introduce a construction that turns a firm monoidal category into a restriction category and axiomatise the monoidal restriction categories that arise this way, called tensor-restriction categories.
\end{abstract}
\maketitle
\allowdisplaybreaks

\section{Introduction}

The tensor product in a monoidal category encodes a notion of space. 
For example, in the category of sheaves over a topological space, as well as in the category of continuous fields of Banach spaces over a topological space, the open sets correspond precisely to so-called \emph{subunits}: subobjects of the tensor unit that are idempotent (in a sense made precise below). 
Under the mild condition that the monoidal category is \emph{firm}, meaning that its subunits are closed under tensor product, we may think about the semilattice they form as a base space underlying the monoidal category. 
Moreover, the tensor product provides methods to deal with partiality, restriction, and support, with respect to this base space. \emph{Tensor topology}~\cite{enriquemolinerheunentull:tensortopology} deals with objects and maps that are partially defined with respect to the tensor structure.\footnote{Tensor topology should not be confused with monoidal topology~\cite{hofmannsealtholen:monoidaltopology}.}

There is another dimension to monoidal categories than the tensor product, namely composition of morphisms. \emph{Restriction categories}~\cite{cockettlack:restrictioncategories} were established to deal with maps that are partially defined with respect to composition. For restriction categories, the elegantly simple main technique is to record the domain of definition of a morphism $f \colon A \to B$ in an endomorphism $\rest{f} \colon A \to A$. Thus each object $A$ has a space $\mathcal{O}(A) = \{ e \colon A \to A \mid \rest{e} = e \}$ underlying it, and there are methods to deal with partiality, restriction, and support, with respect to this base space and in terms of composition.

This article brings the two notions of partiality together. We introduce a construction that turns a firm monoidal category $\cat{C}$ into a restriction category $\S[\cat{C}]$ in a functorial way. We analyse which restriction categories arise this way, and axiomatise them as \emph{tensor-restriction categories}. In the other direction, we will show that the known construction of taking the restriction-total maps turns a tensor-restriction category $\cat{X}$ back into a firm monoidal category $\T[\cat{X}]$. Indeed, $\S[\T[\cat{X}]] \simeq \cat{X}$ and $\T[\S[\cat{C}]] \simeq \cat{C}$, and this gives an equivalence of categories.

This result has advantages to both tensor topology and restriction category theory. On the one hand, restriction categories are relatively more established, and tensor topology is relatively more recent, so one may hope that techniques from restriction category theory can usefully be applied to tensor topology. On the other hand, tensor topology gives a new source of examples of restriction categories, and these are in a sense more naturally dealt with as tensor-restriction categories. 

Another advantage of tensor topology over restriction categories may be that monoidal categories afford a graphical calculus~\cite{heunenvicary:cqm}. We leave to future work the idea of adapting the graphical calculus to subunits, but make some initial remarks in Section~\ref{sec:alternatives}, which may lead to appealing visual methods to deal with partiality and restriction in general. 
More generally, the two dimensions of composition and tensor product are brought together in bicategories. We leave open the question of whether our results have a common generalisation in `restriction bicategories', but
point out that the orthogonal factorisation system of tensor-restriction categories (see Proposition~\ref{prop:factorisation}) strongly resembles the interchange law of bicategories; see also~\cite{cockettheunen:compactinverse}.

We start by recalling the basics of tensor topology in Section~\ref{sec:subunits}. The required notions from restriction category theory are recalled in Section~\ref{sec:Sconstruction}. This section simultaneously introduces the $\S$-construction and analyses it to illustrate these notions. Section~\ref{sec:monoidalrestrictioncats} starts to consider categories that have both a tensor product and a restriction structure, and Section~\ref{sec:tensorrestrictioncats} axiomatises the resulting tensor-restriction categories. We conclude, in Section~\ref{sec:alternatives}, by discussing possible alternative characterisations of tensor-restriction categories. As a matter of terminology, to distinguish between similar notions from tensor topology and restriction categories, we will consistently prefix them, so that a morphism can for example be restriction-total or tensor-total.

\section{Subunits}\label{sec:subunits}

In this section, we recall the notion of a subunit and its properties~\cite{enriquemolinerheunentull:tensortopology}, and draw them in the graphical calculus of monoidal categories~\cite{heunenvicary:cqm}. For a monoidal category $\cat{C}$, denote the monoidal product as $\otimes$, the monoidal unit as $I$, and the coherence natural isomorphisms as $\lambda_A \colon I \otimes A \to A$, $\rho_A\colon A \otimes I \to A$, and $\alpha_{A,B,C}\colon A \otimes (B \otimes C) \to (A \otimes B) \otimes C$. If $\cat{C}$ is braided, write $\sigma_{A,B}\colon A \otimes B \to B \otimes A$ for the braiding. We will often omit subscripts and simply write $\lambda$, $\rho$, $\alpha$, and $\sigma$ when there can be no confusion. We will also abbreviate identity morphisms $1_A\colon A \to A$ simply as $A$. 

First recall that in any category $\cat{C}$, a subobject of an object $A$ is an equivalence class of monomorphisms $s \colon S \rightarrowtail A$, where $s$ is equivalent to $s^\prime \colon S^\prime \rightarrowtail A$ if there is an isomorphism $m \colon S \to S^\prime$ with $s^\prime \circ m = s$. For subunits, we will use a small letter $s$ to denote a representing monomorphism and a capital letter $S$ for its domain. 

\begin{definition}\cite[Definition 2.1]{enriquemolinerheunentull:tensortopology} 
 A \emph{subunit} in a monoidal category $\cat{C}$ is a subobject ${s \colon S \rightarrowtail I}$ of the monoidal unit such that $s \otimes S \colon S \otimes S \to S \otimes I$ is invertible.
 We write $\ISub(\cat{C})$ for the class of subunits in $\cat{C}$.
\end{definition}

Note that for a subunit $s \colon S \rightarrowtail I$, the composite $\rho \circ (s \otimes S) \colon S \otimes S \to S$ is also an isomorphism, so $S \cong S \otimes S$. Furthermore since $s$ is monic and $s \otimes S$ is invertible, then $S \otimes s$ is also invertible. The following is a very useful property of subunits. 

\begin{lemma}\label{lem:SsissS} 
 If $s \colon S \rightarrowtail I$ is a subunit then $\lambda \circ (s \otimes S) = \rho \circ (S \otimes s)$. 
\end{lemma}
\begin{proof} 
 Note that we have the following equality: 
 \begin{align*}
  s \circ \lambda_S \circ (s \otimes S) 
  &= \lambda_I \circ (I \otimes s) \circ (s \otimes S) & \text{(naturality of $\lambda$)} \\
  &= \lambda_I \circ (s \otimes s) \\
  &= \rho_I \circ (s \otimes s) & \text{(coherence)} \\
  &= \rho_I \circ (s \otimes I) \circ (S \otimes s) \\
  &= s \circ \rho_S \circ (S \otimes s) & \text{(naturality of $\rho$)} 
 \end{align*}
 Since $s$ is monic, it follows that $\lambda \circ (s \otimes S) = \rho \circ (S \otimes s)$.
\end{proof}
 
Subunits behave smoothly in a \emph{firm} monoidal category. The following definition means that the subunits are closed under tensor product. It implicitly uses the fact that a subunit $t \colon T \rightarrowtail I$ is completely determined by its domain $T$~\cite[Lemma~2.4]{enriquemolinerheunentull:tensortopology}.

\begin{definition}\cite[Definition 2.5]{enriquemolinerheunentull:tensortopology}
 A braided monoidal category $\cat{C}$ is \emph{firm} when $s \otimes T \colon S \otimes T \to I \otimes T$ is a monomorphism for any subunits $s \colon S \rightarrowtail I$ and $t \colon T \rightarrowtail I$.
\end{definition}

The subunits in a firm monoidal category form a semilattice; as in~\cite{enriquemolinerheunentull:tensortopology}, we will ignore the fact that the subunits may form a proper class rather than a set.

\begin{proposition}\cite[Proposition 2.9]{enriquemolinerheunentull:tensortopology} 
 If $\cat{C}$ is a firm monoidal category, then $\ISub(\cat{C})$ is a semilattice where the meet $\wedge$ of subunits $s \colon S \rightarrowtail I$ and $t \colon T \rightarrowtail I$ is $\left( s \colon S \rightarrowtail I \right) \wedge \left( t \colon T \rightarrowtail I \right) 
  = 
  \left( s \wedge t \colon S \otimes T \rightarrowtail I \right)$ 
 where $s \wedge t = \lambda \circ (s \otimes t)$, and with top element the (equivalence class of the identity on) the monoidal unit $1 \colon I \to I$. The induced order $\leq$ on $\ISub(\cat{C})$ is precisely the order of subunits: $s \leq t$ if and only if $s = t \circ m$ for some $m \colon S \to T$. 
\end{proposition}

To ease definitions and proofs, we will freely use the graphical calculus for monoidal categories. (We invite readers unfamiliar with string diagrams to see~\cite{heunenvicary:cqm} for an introduction). String diagrams in this paper are read from bottom to top. 

Subunits are determined by their domain: if $s,s' \colon S \rightarrowtail I$ are subunits, then they represent the same subobject.
Therefore we may draw a subunit $S \rightarrowtail I$ as:
\[\begin{pic}
 \node[dot] (d) {};
 \draw (d) to +(0,-.5) node[below]{$S$};
\end{pic}\]
In particular, the equality of Lemma \ref{lem:SsissS} is drawn as: 
  \begin{equation}\label{eq:SsissS}
  \begin{pic}
   \node[dot] (d) at (0,0) {};
   \draw (d) to (0,-.5) node[below]{$S$};
   \draw (.3,-.5) node[below]{$S$} to (.3,.5);
      \end{pic}
  =
   \begin{pic}
    \node[dot] (d) at (0,0) {};
    \draw (d) to (0,-.5) node[below]{$S$};
    \draw (-.3,-.5) node[below]{$S$} to (-.3,.5);
   \end{pic}
  \end{equation}
We draw the inverse of $\lambda \circ (s \otimes S) = \rho \circ (S \otimes s)$ as:
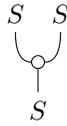
\[\begin{pic}
 \node[dot] (d) at (0,0) {};
 \draw (d) to +(0,-.4) node[below]{$S$};
 \draw (d) to[out=0,in=-90] +(.3,.4) node[above]{$S$};
 \draw (d) to[out=180,in=-90] +(-.3,.4) node[above]{$S$};
\end{pic}\]
and so in particular: 
\[\begin{pic}
 \node[dot] (d) at (0,0) {};
 \node[dot] (d) at (.3,-.4) {};
 \draw (d) to +(0,-.4) node[below]{$S$};
 \draw (d) to[out=0,in=-90] +(.3,.5) node[above]{$S$};
 \draw (d) to[out=180,in=-90] +(-.3,.3) node[above]{};
\end{pic} =   \begin{pic}
   \draw (.3,-.5) node[below]{$S$} to (.3,.5) node[above]{$S$};
      \end{pic} = \begin{pic}
 \node[dot] (d) at (0,0) {};
 \node[dot] (d) at (-.3,-.4) {};
 \draw (d) to +(0,-.4) node[below]{$S$};
 \draw (d) to[out=180,in=-90] +(-.3,.5) node[above]{$S$};
 \draw (d) to[out=0,in=-90] +(.3,.3) node[above]{};
\end{pic} \]

We conclude this section with some running examples of firm monoidal categories and their subunits; for more details and other examples see~\cite[Section 3]{enriquemolinerheunentull:tensortopology}. 

\begin{example} Let $(L, \wedge, 1)$ be a semilattice. Then $L$ can be regarded as a category whose objects are the elements $x \in L$ and where there is a unique $x \to y$ if $x \leq y$. Furthermore, $L$ is a firm monoidal category with $x \otimes y = x \wedge y$ and unit $1$. Every element of $L$ is a subunit, and therefore $\ISub(L) = L$. 
\end{example}

\begin{example} Let $X$ be a locale, with frame of opens $\mathcal{O}(X)$ and top element $\top$, and let $\mathrm{Sh}(X)$ be the category of sheaves over $X$. For every open $U \in \cO(X)$, define the sheaf $\chi_U \colon \cO(X)\op \to \cat{Set}$ as follows: 
 \[
  \chi_U(V) = \begin{cases} 
   \lbrace \ast \rbrace & \text{ if } U \geq V \\
   \emptyset & \text{ if } U \not\geq V
  \end{cases}
 \]
Now $\mathrm{Sh}(X)$ is a firm monoidal category where the monoidal structure is given by finite products, so the monoidal product is given by the pointwise cartesian product $\otimes = \times$ and the monoidal unit is the terminal sheaf $I = \chi_\top$. The subunits of $\mathrm{Sh}(X)$ are given by the opens of $X$, that is, they are precisely the subterminal sheaves $\chi_U$, so $\ISub(\mathrm{Sh}(X)) \cong \mathcal{O}(X)$. 
\end{example}

\begin{example} Let $R$ be a commutative unital ring and $\cat{Mod}_R$ be its category of $R$-modules and $R$-linear morphisms between them. $\cat{Mod}_R$ is a firm monoidal category where the monoidal structure is given by the standard tensor product of $R$-modules and so the monoidal unit is the ring itself $I = R$. The subunits of $\cat{Mod}_R$ are the idempotent ideals of $R$, that is, ideals $J \subseteq R$ such that $J = J^2$, where $J^2 = \{ \sum_{i=1}^n r_i r'_i \mid r_i,r_i' \in J \}$. 
\end{example}

\begin{example} 
 Recall that a \emph{Boolean ring} is a commutative unital ring $R$ such that $x^2 = x$ for all $x \in R$. Every ideal of a boolean ring is idempotent. Therefore, for a Boolean ring $R$, the subunits of $\cat{Mod}_R$ correspond precisely to the ideals of $R$. Equivalently, $R$ is a Boolean algebra, and $\ISub(\cat{Mod}_R)$ consists of its order ideals.
\end{example}

\begin{example} 
  The category $\cat{Set}$ of sets and functions is a firm monoidal category under the Cartesian product with monoidal unit a chosen singleton set $I = \{\ast \}$. There are only two subunits in $\cat{Set}$: the empty set $\emptyset$ and the singleton $\{\ast \}$. 
\end{example}

\begin{example}
  Let $R$ be a commutative semiring, and consider the category $\cat{SMod}_R$ of $R$-semimodules~\cite[Chapters~14 and~16]{golan:semirings}. 
  There are always two subunits in $\cat{SMod}_R$: $R$ itself, and the zero object $\mathsf{0}=\{0\}$. Similar to the previous example, if these are the only two subunits, then $\cat{SMod}_R$ is a firm monoidal category. This situation includes the case where $R$ is a semifield, such as the category $\cat{Vect}_k$ of vector spaces over a field $k$, and the category $\cat{Rel}$ of sets and relations where the semiring is that of Boolean truth values~\cite[Definition~0.5]{heunenvicary:cqm}.
\end{example}

\section{Restriction Categories from Subunits}\label{sec:Sconstruction}

This section introduce the $\S[-]$-construction, that turns a firm monoidal category $\cat{C}$ into a restriction category $\S[\cat{C}]$ where the restriction structure is determined by the subunits of $\cat{C}$. We will recall various important notions of restriction categories, such as restriction idempotents, restriction-total morphisms, and restriction isomorphisms, and study them in $\S[\cat{C}]$. (For a more in-depth introduction and more details on restriction categories, see~\cite{cockettlack:restrictioncategories}.)

We begin with the $\S[-]$ construction itself.

\begin{definition}
 Let $\cat{C}$ be a firm monoidal category. Define a category $\S[\cat{C}]$ as follows: 
 \begin{itemize}
  \item Objects are the same as in $\cat{C}$;
  \item Morphisms $[s,f] \colon A \to B$ in $\S[\cat{C}]$ are equivalence classes of pairs $(s,f)$ of a subunit $s \colon S \rightarrowtail I$ and a morphism $f \colon A \otimes S \to B$ in $\cat{C}$, where pairs $(s,f)$ and $(s',f')$ are identified when $s'=s \circ m$ for an isomorphism $m \colon S' \to S$ and $f'=f \circ (A \otimes m)$. We will draw morphisms graphically:
  \[ [s,f] = \Big[s,\;\begin{minipic}
    \node[morphism] (f) at (0,0) {$f$};
    \draw ([xshift=-1mm]f.south west) to +(0,-.2) node[right=-1mm,font=\tiny]{$A$};
    \draw ([xshift=1mm]f.south east) to +(0,-.2) node[right=-1mm,font=\tiny]{$S$};
    \draw (f.north) to +(0,.15) node[right=-1mm,font=\tiny]{$B$};
   \end{minipic}\Big] \]
  \item Identity morphisms are $[\id[I], \rho] = [\id[I],\begin{minipic}[baseline={(0,.3)}]
   \draw (0,.4) to (0,0) node[right=-1mm,font=\tiny]{$A$};
  \end{minipic}] \colon A \to A$
  \item Composition of $\Big[s,\;\begin{minipic}
    \node[morphism] (f) at (0,0) {$f$};
    \draw ([xshift=-1mm]f.south west) to +(0,-.2) node[right=-1mm,font=\tiny]{$A$};
    \draw ([xshift=1mm]f.south east) to +(0,-.2) node[right=-1mm,font=\tiny]{$S$};
    \draw (f.north) to +(0,.15) node[right=-1mm,font=\tiny]{$B$};
   \end{minipic}\!\Big]$ and $ \Big[t,\;\begin{minipic}
    \node[morphism] (f) at (0,0) {$g$};
    \draw ([xshift=-1mm]f.south west) to +(0,-.2) node[right=-1mm,font=\tiny]{$B$};
    \draw ([xshift=1mm]f.south east) to +(0,-.2) node[right=-1mm,font=\tiny]{$T$};
    \draw (f.north) to +(0,.15) node[right=-1mm,font=\tiny]{$C$};
   \end{minipic}\!\Big]$ is defined as follows: 
  \[    
   \Big[t,\;\begin{minipic}
    \node[morphism] (f) at (0,0) {$g$};
    \draw ([xshift=-1mm]f.south west) to +(0,-.2) node[right=-1mm,font=\tiny]{$B$};
    \draw ([xshift=1mm]f.south east) to +(0,-.2) node[right=-1mm,font=\tiny]{$T$};
    \draw (f.north) to +(0,.15) node[right=-1mm,font=\tiny]{$C$};
   \end{minipic}\!\Big]
   \circ \Big[s,\;\begin{minipic}
    \node[morphism] (f) at (0,0) {$f$};
    \draw ([xshift=-1mm]f.south west) to +(0,-.2) node[right=-1mm,font=\tiny]{$A$};
    \draw ([xshift=1mm]f.south east) to +(0,-.2) node[right=-1mm,font=\tiny]{$S$};
    \draw (f.north) to +(0,.15) node[right=-1mm,font=\tiny]{$B$};
   \end{minipic}\!\Big] = 
   \Bigg[s \wedge t,\;\;\begin{minipic}
    \node[morphism,width=5mm] (f) at (0,0) {$f$};
    \node[morphism,width=5mm,anchor=south west] (g) at ([xshift=1mm,yshift=2mm]f.north) {$g$};
    \draw (f.north) to ([xshift=-1mm]g.south west);
    \draw (f.south west) to +(0,-.3) node[right=-1mm,font=\tiny]{$A$};
    \draw (f.south east) to +(0,-.3) node[right=-1mm,font=\tiny]{$S$};
    \draw ([xshift=1mm]g.south east) to +(0,-.8) node[right=-1mm,font=\tiny]{$T$};
    \draw (g.north) to +(0,.2) node[right=-1mm,font=\tiny]{$C$};
   \end{minipic}\Bigg]
  \]
 \end{itemize}
\end{definition}

It is straightforward to check that $\S[\cat{C}]$ is indeed a well-defined category when $\cat{C}$ is a firm monoidal category.
Our next objective is to explain how $\S[\cat{C}]$ is in fact a restriction category in such a way that the restriction of $\Big[s,\;\smash{\begin{minipic}
    \node[morphism] (f) at (0,0) {$f$};
    \draw ([xshift=-1mm]f.south west) to +(0,-.2) node[right=-1mm,font=\tiny]{$A$};
    \draw ([xshift=1mm]f.south east) to +(0,-.2) node[right=-1mm,font=\tiny]{$S$};
    \draw (f.north) to +(0,.15) node[right=-1mm,font=\tiny]{$B$};
   \end{minipic}}\!\Big]$ is determined by the subunit $s$. First recall the definition of a restriction category. 

\begin{definition}\cite[Section 2.1.1]{cockettlack:restrictioncategories}
 A \emph{restriction category} is a category $\cat{X}$ equipped with a choice of endomorphism $\rest{f} \colon A \to A$ for each morphism $f \colon A \to B$ satisfying:
\begin{align}  
 f \circ \rest{f} & = f \tag{R1} \\
 \rest{f} \circ \rest{g} & = \rest{g} \circ \rest{f} &&\text{if $\dom(f)=\dom(g)$} \tag{R2} \\
 \rest{g \circ \rest{f}} & = \rest{g} \circ \rest{f} &&\text{if $\dom(f)=\dom(g)$} \tag{R3} \\
 \rest{g} \circ f & = f \circ \rest{g \circ f} &&\text{if $\cod(f)=\dom(g)$} \tag{R4}
\end{align}
 We call $\rest{f}$ the \emph{restriction} of $f$.
\end{definition}

\begin{example} The canonical example of a restriction category is $\cat{Par}$, the category of sets and partial functions. The restriction $\rest{f}\colon X \to X$ of a partial function $f \colon X \to Y$ is
 \begin{align*}
\rest{f}(x) = \begin{cases} 
x & \text{ if $f(x)$ is defined} \\
 \uparrow & \text{ otherwise } 
  \end{cases}
 \end{align*}
 where $\uparrow$ means ``undefined''. 
\end{example}

For more examples of restriction categories, see~\cite[Section 2.1.3]{cockettlack:restrictioncategories}. We move to the restriction structure of $\S[\cat{C}]$.

\begin{proposition}\label{prop:SCrestriction}
 If $\cat{C}$ is a firm monoidal category, then $\S[\cat{C}]$ is a restriction category with:
 \[
   \rest{ \Big[s,\;\begin{minipic}
    \node[morphism] (f) at (0,0) {$f$};
    \draw ([xshift=-1mm]f.south west) to +(0,-.2) node[right=-1mm,font=\tiny]{$A$};
    \draw ([xshift=1mm]f.south east) to +(0,-.2) node[right=-1mm,font=\tiny]{$S$};
    \draw (f.north) to +(0,.15) node[right=-1mm,font=\tiny]{$B$};
   \end{minipic}\!\Big]} = 
  \Big[s,\,\begin{minipic}
   \draw (0,0) to (0,-.5) node[right=-1mm,font=\tiny]{$A$};
   \draw (.3,-.2) node[dot]{} to (.3,-.5) node[right=-1mm,font=\tiny]{$S$};
  \end{minipic}\!\Big]
 \]
\end{proposition}
\begin{proof}
 We start with verifying (R1):
 \begin{align*}
  \Big[s,\;\begin{minipic}
    \node[morphism] (f) at (0,0) {$f$};
    \draw ([xshift=-1mm]f.south west) to +(0,-.2) node[right=-1mm,font=\tiny]{$A$};
    \draw ([xshift=1mm]f.south east) to +(0,-.2) node[right=-1mm,font=\tiny]{$S$};
    \draw (f.north) to +(0,.15) node[right=-1mm,font=\tiny]{$B$};
   \end{minipic}\!\Big] \circ \rest{\Big[s,\;\begin{minipic}
    \node[morphism] (f) at (0,0) {$f$};
    \draw ([xshift=-1mm]f.south west) to +(0,-.2) node[right=-1mm,font=\tiny]{$A$};
    \draw ([xshift=1mm]f.south east) to +(0,-.2) node[right=-1mm,font=\tiny]{$S$};
    \draw (f.north) to +(0,.15) node[right=-1mm,font=\tiny]{$B$};
   \end{minipic}\!\Big]}
  & = \Big[s,\;\begin{minipic}
   \node[morphism] (f) at (0,0) {$f$};
   \draw ([xshift=-1mm]f.south west) to +(0,-.2) node[right=-1mm,font=\tiny]{$A$};
   \draw ([xshift=1mm]f.south east) to +(0,-.2) node[right=-1mm,font=\tiny]{$S$};
   \draw (f.north) to +(0,.2) node[right=-1mm,font=\tiny]{$B$};
  \end{minipic}\!\Big] \circ \Big[s,\,\begin{minipic}[baseline={(0,.3)}]
   \draw (0,.4) to (0,-.2) node[right=-1mm,font=\tiny]{$A$};
   \draw (.3,-.2) node[right=-1mm,font=\tiny]{$S$} to (.3,.2) node[dot]{};
  \end{minipic}\!\Big] 
  \\
  & = \Big[s \wedge s,\;\begin{minipic}
   \node[morphism] (f) at (0,0) {$f$};
   \draw ([xshift=-1mm]f.south west) to +(0,-.2) node[right=-1mm,font=\tiny]{$A$};
   \draw ([xshift=1mm]f.south east) to +(0,-.2) node[right=-1mm,font=\tiny]{$S$};
   \draw (f.north) to +(0,.2) node[right=-1mm,font=\tiny]{$B$};
   \draw (.5,-.35) node[right=-1mm,font=\tiny]{$S$} to (.5,0) node[dot]{};
  \end{minipic}\!\Big] 
  = \bigg[s,\;\begin{minipic}
   \node[morphism] (f) at (0,0) {$f$};
   \draw ([xshift=-1mm]f.south west) to +(0,-.4) node[right=-1mm,font=\tiny]{$A$};
   \draw (f.north) to +(0,.2) node[right=-1mm,font=\tiny]{$B$};
   \node[dot] (d) at (.4,-.35) {};
   \draw ([xshift=1mm]f.south east) to[out=-90,in=180] (d);
   \draw (.6,0) node[dot]{} to[out=-90,in=0,looseness=.8] (d);
   \draw (d) to +(0,-.2) node[right=-1mm,font=\tiny]{$S$};
  \end{minipic}\bigg] 
  = \Big[s,\;\begin{minipic}
    \node[morphism] (f) at (0,0) {$f$};
    \draw ([xshift=-1mm]f.south west) to +(0,-.2) node[right=-1mm,font=\tiny]{$A$};
    \draw ([xshift=1mm]f.south east) to +(0,-.2) node[right=-1mm,font=\tiny]{$S$};
    \draw (f.north) to +(0,.15) node[right=-1mm,font=\tiny]{$B$};
   \end{minipic}\!\Big]
 \end{align*}
 Next, (R2):
 \[
  \rest{\Big[s,\;\begin{minipic}
    \node[morphism] (f) at (0,0) {$f$};
    \draw ([xshift=-1mm]f.south west) to +(0,-.2) node[right=-1mm,font=\tiny]{$A$};
    \draw ([xshift=1mm]f.south east) to +(0,-.2) node[right=-1mm,font=\tiny]{$S$};
    \draw (f.north) to +(0,.15) node[right=-1mm,font=\tiny]{$B$};
   \end{minipic}\!\Big]} \circ \rest{ \Big[t,\;\begin{minipic}
    \node[morphism] (f) at (0,0) {$g$};
    \draw ([xshift=-1mm]f.south west) to +(0,-.2) node[right=-1mm,font=\tiny]{$A$};
    \draw ([xshift=1mm]f.south east) to +(0,-.2) node[right=-1mm,font=\tiny]{$T$};
    \draw (f.north) to +(0,.15) node[right=-1mm,font=\tiny]{$C$};
   \end{minipic}\!\Big]}   
  = \Big[s,\,\begin{minipic}[baseline={(0,.4)}]
   \draw (0,0) node[right=-1mm,font=\tiny]{$A$} to (0,.5);
   \draw (.3,0) node[right=-1mm,font=\tiny]{$S$} to (.3,.3) node[dot]{};
  \end{minipic}\!\Big] \circ \Big[t,\,\begin{minipic}[baseline={(0,.4)}]
   \draw (0,0) node[right=-1mm,font=\tiny]{$A$} to (0,.5);
   \draw (.3,0) node[right=-1mm,font=\tiny]{$T$} to (.3,.3) node[dot]{};
  \end{minipic}\!\Big]
  = \Big[s \wedge t,\,\begin{minipic}[baseline={(0,.4)}]
   \draw (0,0) node[right=-1mm,font=\tiny]{$A$} to (0,.5);
   \draw (.3,0) node[right=-1mm,font=\tiny]{$S$} to (.3,.3) node[dot]{};
   \draw (.6,0) node[right=-1mm,font=\tiny]{$T$} to (.6,.3) node[dot]{};
  \end{minipic}\!\Big]
  = \rest{ \Big[t,\;\begin{minipic}
    \node[morphism] (f) at (0,0) {$g$};
    \draw ([xshift=-1mm]f.south west) to +(0,-.2) node[right=-1mm,font=\tiny]{$A$};
    \draw ([xshift=1mm]f.south east) to +(0,-.2) node[right=-1mm,font=\tiny]{$T$};
    \draw (f.north) to +(0,.15) node[right=-1mm,font=\tiny]{$C$};
   \end{minipic}\!\Big]} \circ \rest{\Big[s,\;\begin{minipic}
    \node[morphism] (f) at (0,0) {$f$};
    \draw ([xshift=-1mm]f.south west) to +(0,-.2) node[right=-1mm,font=\tiny]{$A$};
    \draw ([xshift=1mm]f.south east) to +(0,-.2) node[right=-1mm,font=\tiny]{$S$};
    \draw (f.north) to +(0,.15) node[right=-1mm,font=\tiny]{$B$};
   \end{minipic}\!\Big]}   
 \]
 For (R3):
 \begin{align*}
  \rest{ \Big[t,\;\begin{minipic}
    \node[morphism] (f) at (0,0) {$g$};
    \draw ([xshift=-1mm]f.south west) to +(0,-.2) node[right=-1mm,font=\tiny]{$A$};
    \draw ([xshift=1mm]f.south east) to +(0,-.2) node[right=-1mm,font=\tiny]{$T$};
    \draw (f.north) to +(0,.15) node[right=-1mm,font=\tiny]{$C$};
   \end{minipic}\!\Big] \circ \rest{\Big[s,\;\begin{minipic}
    \node[morphism] (f) at (0,0) {$f$};
    \draw ([xshift=-1mm]f.south west) to +(0,-.2) node[right=-1mm,font=\tiny]{$A$};
    \draw ([xshift=1mm]f.south east) to +(0,-.2) node[right=-1mm,font=\tiny]{$S$};
    \draw (f.north) to +(0,.15) node[right=-1mm,font=\tiny]{$B$};
   \end{minipic}\!\Big]}}
  &= \rest{
   \Big[t,\;\begin{minipic}
    \node[morphism] (f) at (0,0) {$g$};
    \draw ([xshift=-1mm]f.south west) to +(0,-.2) node[right=-1mm,font=\tiny]{$A$};
    \draw ([xshift=1mm]f.south east) to +(0,-.2) node[right=-1mm,font=\tiny]{$T$};
    \draw (f.north) to +(0,.15) node[right=-1mm,font=\tiny]{$C$};
   \end{minipic}\!\Big]
   \circ
   \Big[s,\,\begin{minipic}[baseline={(0,.5)}]
    \draw (0,0) node[right=-1mm,font=\tiny]{$A$} to (0,.5);
    \draw (.3,0) node[right=-1mm,font=\tiny]{$S$} to (.3,.3) node[dot]{};
   \end{minipic}\!\Big]
  } \\
  &= \rest{\Big[t \wedge s,\;\begin{minipic}
    \node[morphism] (f) at (0,0) {$g$};
    \draw ([xshift=-1mm]f.south west) to +(0,-.2) node[right=-1mm,font=\tiny]{$A$};
    \draw ([xshift=1mm]f.south east) to +(0,-.2) node[right=-1mm,font=\tiny]{$T$};
    \draw (f.north) to +(0,.15) node[right=-1mm,font=\tiny]{$C$};
    \draw (.5,-.35) node[right=-1mm,font=\tiny]{$S$} to (.5,0) node[dot]{};
   \end{minipic}\!\Big]} \\
 & = \Big[s \wedge t,\,\begin{minipic}
   \draw (0,0) node[right=-1mm,font=\tiny]{$A$} to (0,.5);
   \draw (.3,0) node[right=-1mm,font=\tiny]{$S$} to (.3,.3) node[dot]{};
   \draw (.6,0) node[right=-1mm,font=\tiny]{$T$} to (.6,.3) node[dot]{};
  \end{minipic}\!\!\Big] \\
 &  = \rest{ \Big[t,\;\begin{minipic}[baseline={(0,.5)}]
    \node[morphism] (f) at (0,0) {$g$};
    \draw ([xshift=-1mm]f.south west) to +(0,-.2) node[right=-1mm,font=\tiny]{$A$};
    \draw ([xshift=1mm]f.south east) to +(0,-.2) node[right=-1mm,font=\tiny]{$T$};
    \draw (f.north) to +(0,.15) node[right=-1mm,font=\tiny]{$C$};
   \end{minipic}\!\Big]} \circ \rest{\Big[s,\;\begin{minipic}
    \node[morphism] (f) at (0,0) {$f$};
    \draw ([xshift=-1mm]f.south west) to +(0,-.2) node[right=-1mm,font=\tiny]{$A$};
    \draw ([xshift=1mm]f.south east) to +(0,-.2) node[right=-1mm,font=\tiny]{$S$};
    \draw (f.north) to +(0,.15) node[right=-1mm,font=\tiny]{$B$};
   \end{minipic}\!\Big]}
 \end{align*}
 Finally, (R4):
 \begin{align*}
  \Big[s,\;\begin{minipic}
    \node[morphism] (f) at (0,0) {$f$};
    \draw ([xshift=-1mm]f.south west) to +(0,-.2) node[right=-1mm,font=\tiny]{$A$};
    \draw ([xshift=1mm]f.south east) to +(0,-.2) node[right=-1mm,font=\tiny]{$S$};
    \draw (f.north) to +(0,.15) node[right=-1mm,font=\tiny]{$B$};
   \end{minipic}\!\Big] \circ \rest{ \Big[t,\;\begin{minipic}
    \node[morphism] (f) at (0,0) {$g$};
    \draw ([xshift=-1mm]f.south west) to +(0,-.2) node[right=-1mm,font=\tiny]{$B$};
    \draw ([xshift=1mm]f.south east) to +(0,-.2) node[right=-1mm,font=\tiny]{$T$};
    \draw (f.north) to +(0,.15) node[right=-1mm,font=\tiny]{$C$};
   \end{minipic}\!\Big] \circ \Big[s,\;\begin{minipic}
    \node[morphism] (f) at (0,0) {$f$};
    \draw ([xshift=-1mm]f.south west) to +(0,-.2) node[right=-1mm,font=\tiny]{$A$};
    \draw ([xshift=1mm]f.south east) to +(0,-.2) node[right=-1mm,font=\tiny]{$S$};
    \draw (f.north) to +(0,.15) node[right=-1mm,font=\tiny]{$B$};
   \end{minipic}\!\Big]}
  & = \Big[s,\;\begin{minipic}
    \node[morphism] (f) at (0,0) {$f$};
    \draw ([xshift=-1mm]f.south west) to +(0,-.2) node[right=-1mm,font=\tiny]{$A$};
    \draw ([xshift=1mm]f.south east) to +(0,-.2) node[right=-1mm,font=\tiny]{$S$};
    \draw (f.north) to +(0,.15) node[right=-1mm,font=\tiny]{$B$};
   \end{minipic}\!\Big] \circ \rest{\Bigg[s \wedge t,\,\begin{minipic}
    \node[morphism,width=4mm] (f) at (0,0) {$f$};
    \node[morphism,width=4mm,anchor=south west] (g) at ([xshift=1mm,yshift=1.5mm]f.north) {$g$};
    \draw (f.north) to ([xshift=-1mm]g.south west);
    \draw (f.south west) to +(0,-.2) node[right=-1mm,font=\tiny]{$A$};
    \draw (f.south east) to +(0,-.2) node[right=-1mm,font=\tiny]{$S$};
    \draw ([xshift=1mm]g.south east) to +(0,-.65) node[right=-1mm,font=\tiny]{$T$};
    \draw (g.north) to +(0,.15);
   \end{minipic}\Bigg]} \\
  & = \Big[s,\;\begin{minipic}
    \node[morphism] (f) at (0,0) {$f$};
    \draw ([xshift=-1mm]f.south west) to +(0,-.2) node[right=-1mm,font=\tiny]{$A$};
    \draw ([xshift=1mm]f.south east) to +(0,-.2) node[right=-1mm,font=\tiny]{$S$};
    \draw (f.north) to +(0,.15);
   \end{minipic}\!\Big] \circ \Big[s \wedge t,\,\begin{minipic}[baseline={(0,.4)}]
    \draw (0,0) node[right=-1mm,font=\tiny]{$A$} to (0,.5);
    \draw (.3,0) node[right=-1mm,font=\tiny]{$S$} to (.3,.3) node[dot]{};
    \draw (.6,0) node[right=-1mm,font=\tiny]{$T$} to (.6,.3) node[dot]{};
   \end{minipic}\!\Big] \\
  & = \Big[s \wedge s \wedge t,\;\,\begin{minipic}
    \node[morphism] (f) at (0,0) {$f$};
    \draw ([xshift=-1mm]f.south west) to +(0,-.2) node[right=-1mm,font=\tiny]{$A$};
    \draw ([xshift=1mm]f.south east) to +(0,-.2) node[right=-1mm,font=\tiny]{$S$};
    \draw (f.north) to +(0,.15);
    \draw (.5,-.35) node[right=-1mm,font=\tiny]{$S$} to (.5,0) node[dot]{};
    \draw (.85,-.35) node[right=-1mm,font=\tiny]{$T$} to (.85,0) node[dot]{};
   \end{minipic}\!\Big] \\
   &= \Big[s \wedge t,\;\,\begin{minipic}
    \node[morphism] (f) at (0,0) {$f$};
    \draw ([xshift=-1mm]f.south west) to +(0,-.2) node[right=-1mm,font=\tiny]{$A$};
    \draw ([xshift=1mm]f.south east) to +(0,-.2) node[right=-1mm,font=\tiny]{$S$};
    \draw (f.north) to +(0,.15);
    \draw (.5,-.35) node[right=-1mm,font=\tiny]{$T$} to (.5,0) node[dot]{};
   \end{minipic}\!\Big] \\
   &= \rest{ \Big[t,\;\begin{minipic}
    \node[morphism] (f) at (0,0) {$g$};
    \draw ([xshift=-1mm]f.south west) to +(0,-.2) node[right=-1mm,font=\tiny]{$B$};
    \draw ([xshift=1mm]f.south east) to +(0,-.2) node[right=-1mm,font=\tiny]{$T$};
    \draw (f.north) to +(0,.15) node[right=-1mm,font=\tiny]{$C$};
   \end{minipic}\!\Big]} \circ \Big[s,\;\begin{minipic}
    \node[morphism] (f) at (0,0) {$f$};
    \draw ([xshift=-1mm]f.south west) to +(0,-.2) node[right=-1mm,font=\tiny]{$A$};
    \draw ([xshift=1mm]f.south east) to +(0,-.2) node[right=-1mm,font=\tiny]{$S$};
    \draw (f.north) to +(0,.15) node[right=-1mm,font=\tiny]{$B$};
   \end{minipic}\!\Big]   
 \end{align*}
 So we conclude that $\S[\cat{C}]$ is a restriction category. 
\end{proof}

We now apply the $\S[-]$ construction to our examples of firm monoidal categories from the previous section and discuss their restriction structure. Most examples are extensions of the following general principle: $\S[\cat{C} \times \cat{D}] \simeq \S[\cat{C}] \times \S[\cat{D}]$ for firm monoidal categories $\cat{C}$ and $\cat{D}$.

\begin{example}\label{ex:semilattices}
 Regard a semilattice $(L, \wedge, 1)$ as a firm monoidal category. Then $\S[L]$ can be described as follows: 
 \begin{itemize}
  \item Objects are elements $x \in L$;
  \item Morphisms $x \to y$ are elements $s \in L$ such that $x \wedge s \leq y$;
  \item Identity morphisms are $1 \colon x \to x$;
  \item Composition of $s \colon x \to y$ and $t \colon y \to z$ is $s \wedge t \colon x \to z$;
  \item Restriction of $s \colon x \to y$ is $s \colon x \to x$.
 \end{itemize}
 This is a typed version of the known construction of \emph{depressing downsets}: the homset $\S[L](b,a)$ is denoted as $a \downarrow b$ in~\cite[Section 2.7]{ballpultrwayland:dedekindmacneille}.
 If $L$ is an implicative semilattice, also called a Heyting semilattice, that is, if $L$ is closed as a monoidal category, then $\S[L](a,b) = \downset (a \multimap b)$.
 Note that $\S[L]$ is enriched over posets.
 In fact, every homset again has finite meets, but may not have a top element. If $L$ is pro-Heyting ~\cite[Definition 3.1]{ballpultrwayland:dedekindmacneille}, that is if $s<t \in L$ implies that the subset $\{r \in L \mid r \wedge t \leq s\}$ equals the set of upper bounds of the set of lower bounds of itself, then $L$ is implicative if and only if every homset has a top element~\cite[Observation 3.2]{ballpultrwayland:dedekindmacneille}. In this sense, the $\S[L]$ construction ``approximates'' the construction of the free Heyting algebra on a semilattice $L$.
\end{example}

\begin{example}\label{ex:sheaves}Let $X$ be a locale. If $F,G \in \mathrm{Sh}(X)$ are sheaves, then a morphism $F \to G$ in $\S[\mathrm{Sh}(X)]$ consists of an open $U \in \cO(X)$ and a natural transformation $\alpha \colon F \times \chi_U \Rightarrow G$. 
Now let $V \in \cO(X)$ be an open such that $V \not\subseteq U$. Then $(F \times \chi_U)(V) = \emptyset$, and so $\alpha_V$ can only be the empty function $\emptyset \to G(V)$.
 In other words, $\alpha$ is completely determined by $\alpha|_U \colon F|_U \Rightarrow G|_U$, where we write $F|_U \colon \cO(U)\op \to \cat{Set}$ for the restriction of the sheaf $F$ to $U$ given by $F|_U(V)=F(V)$ for $U \geq V$.
Therefore, $\S[\mathrm{Sh}(X)]$ can be described as: 
 \begin{itemize}
  \item Objects are sheaves $F \colon \cO(X)\op \to \cat{Set}$;
  \item Morphisms $[U,\alpha]: F \to G$ are pairs of an open $U \in \cO(X)$ and a natural transformation $\alpha \colon F|_U \to G|_U$;
  \item Identity morphisms are the pairs $[X, \id: F \rightarrow F]: F \to F$ (since $F = F|_X$);
  \item Composition of $[U,\alpha]: F \to G$ and $[V,\beta]: G \to H$ is $[V,\beta] \circ [U,\alpha] = [U \wedge V, \beta|_U \circ \alpha|_V]$, where we write $\alpha|_V$ for the restriction of $\alpha \colon F|_U \Rightarrow G|_U$ to $F|_{U \wedge V} \Rightarrow G|_{U \wedge V}$;
  \item Restriction of $[U,\alpha]: F \to G$ is $\rest{[U,\alpha]} = [U,F|_U]$.
 \end{itemize}
 In short: for categories of sheaves $\mathrm{Sh}(X)$, the $\S[-]$-construction makes the morphisms partial in recording an open domain of definition. 
\end{example}

\begin{example} 
 Let $R$ be a unital commutative ring. Then $\S[\cat{Mod}_R]$ can be described as follows:
 \begin{itemize}
  \item Objects are $R$-modules $A$;
  \item Morphisms $[I,f] \colon A \to B$ are pairs of an idempotent ideal $I$ and an $R$-linear map $f \colon A \otimes_R I \to B$;
  \item Identity morphisms are the pairs $[R, A \otimes_R R \simeq A]$;
  \item Composition of $[I,f]$ and $[J,g]$ is $[IJ, g \circ (f \otimes_R J) \colon A \otimes_R I \otimes_R J \to C]$;
  \item Restriction of $[I,f] \colon A \to B$ is $\rest{[I,f]} = [I, \rho \colon A \otimes_R I \to A]$ with $\rho(a \otimes i)=ai$.
 \end{itemize}
 Thus, for categories of modules $\cat{Mod}_R$, the $\S[-]$-construction makes the morphisms partial in recording an ideal in the domain where a morphism acts on. This is closely related to the idea of localisation in algebra.
 For example, if the ring $R$ is semisimple, then any ideal is generated by an idempotent element, so $\ISub(\cat{Mod}_R)$ is the Boolean algebra of idempotent elements. Morphisms in $\S[\cat{Mod}_R]$ are thus pairs of an idempotent $e^2=e \in R$ and a morphism $f \colon eA \to B$.
\end{example}


\begin{example} 
 The previous two examples are related.
 For example, if $R$ is a Boolean ring, then $\cat{Mod}_R$ is equivalent to the category of sheaves of $\mathbb{Z}_2$-vector spaces over the Stone space of $R$ by Pierce's representation theorem~\cite[Chapter V.2]{johnstone:stonespaces}. Ideals of the Boolean ring $R$ correspond to open sets of its Stone space.
\end{example}

\begin{example}\label{ex:sset} 
  Recall that in $\cat{Set}$, the empty set $\emptyset$ is an initial object and so for every set $X$ there is a unique function $\varnothing \colon \emptyset \to X$, called the empty function. On the other hand, the only function whose codomain is $\emptyset$ is the identity function $\id[\emptyset]$. Also $X \times \emptyset = \emptyset$ for any set $X$. So $\S[\cat{Set}]$ is described as follows: 
 \begin{itemize}
  \item Objects are sets $X$;
  \item Morphisms include the usual functions $[1,f] \colon X \to Y$ as well as an extra map $[\emptyset, \varnothing] \colon X \to Y$, so $\S[\cat{Set}](X,Y) \simeq \cat{Set}(X,Y) + 1$;
  \item Identity morphisms are the pairs $[1, X \times \{ \ast \} \simeq X]$;
  \item Composition of $[s,f]$ and $[t,g]$ can be described in the following three cases: 
  if $s=t=1$, then $[1,g] \circ [1,f] = [1, g \circ f]$;
  if $s=\emptyset$, then $f=\varnothing$ and $[t,g] \circ [\emptyset, \varnothing] = [\emptyset, \varnothing_Y]$;
  if $t=\emptyset$, then $g=\varnothing$ and $[\emptyset, \varnothing] \circ [s,f] = [0,0]$. 
  \item Restriction of $[1,f]$ is the identity $\rest{[1,f]} = [1, X \times \{ \ast \} \simeq X]$, while the restriction of $[\emptyset, \varnothing]$ is itself $\rest{[\emptyset, \varnothing]} = [\emptyset, \varnothing]$.
  \end{itemize}
  Recall that a category $\cat{C}$ is said to have \emph{zero morphisms} if there is a family of morphisms $z \colon A \to B$ (for every pair of objects) which are absorbing in the sense that $z \circ f = z$ and $f \circ z = z$. When they exist, zero morphisms are unique. Now $\S[\cat{Set}]$ has zero morphisms $[\emptyset, \varnothing]$. In fact, $\S[\cat{Set}]$ is the free category with zero morphisms over $\cat{Set}$ with respect to the functor $U \colon \cat{Set} \to \S[\cat{Set}]$ defined on objects as $U(A) = A$ and on maps as $U(f) = [1,f]$. So given a category $\cat{C}$ with zero morphisms $z$ and a functor $F \colon \cat{Set} \to \cat{C}$, there exists a unique functor $G \colon \S[\cat{Set}] \to \cat{C}$ which preserves zero morphisms, that is, $G(z) = [\emptyset, \varnothing]$, and satisfying $G \circ U = F$. Explicitly, $G$ is defined on objects as $G(X) = F(X)$ and on morphisms as $G([1,f]) = F(f)$ and $G([\emptyset, \varnothing]) = z$. While $\emptyset$ is an initial object in $\cat{Set}$, $\emptyset$ is no longer initial in $\S[\cat{Set}]$; it is a terminal object instead. 
\end{example}

\begin{example} Let $R$ be a commutative semiring, such that only subunits in $\cat{SMod}_R$ are $R$ and $\mathsf{0}$. Then similar to the previous example, $\S[\cat{SMod}_R]$ can be described as the free category with zero morphisms over $\cat{SMod}_R$ where the zero morphisms are $[\mathsf{0},0]$. However, while $\cat{SMod}_R$ already had zero morphisms $0: A \to B$, the morphisms $[1, 0]: A \to B$ in $\S[\cat{SMod}_R]$ are no longer zero morphisms. Similarly, while $\mathsf{0}$ was a zero object in $\cat{SMod}_R$, this is no longer the case in $\S[\cat{SMod}_R]$; instead $\mathsf{0}$ becomes a restriction-terminal object~\cite[Definition~2.16]{cockettcruttwellgallagher:differentialrestriction}. 
\end{example}

\begin{remark}
  The $\S[-]$-construction resembles the construction of the free restriction category on a fibration of semilattices~\cite{cockettguo:semilatticefibrations}. In our case, the semilattice over each object is the same, namely $\ISub(\cat{C})$. But $\S[\cat{C}]$ does not fit neatly in that framework, for it is not what is called unitary: it is not the case that $[s,f]=[t,g]$ as soon as $s=t$ and $f \otimes r = g \otimes r$ for some $r$.
\end{remark}

We return to studying the restriction structure of the $\S[-]$-construction, by taking a closer look at various classes of maps that are important in restriction category theory.

\begin{definition}\cite[Section 2.1.1]{cockettlack:restrictioncategories}
 A \emph{restriction idempotent} is an endomorphism $e$  in a restriction category with $e = \rest{e}$. Write $\mathcal{O}(A)$ for the set of all restriction idempotents of type $A \to A$. 
\end{definition}

\begin{example} 
 In $\cat{Par}$, the restriction idempotents of a set $X$ correspond precisely to its subsets $U \subseteq X$. Indeed, for every subset $U \subseteq X$, define the partial function $\chi_U \colon X \to X$ as follows: 
 \begin{align*}
  \chi_U(x) = \begin{cases} 
   x & \text{ if } x \in U \\
   \text{undefined} & \text{ if } x \notin U 
  \end{cases}
   \end{align*}
 Then clearly $\rest{\chi_U} = \chi_U$. Conversly, given a restriction idempotent $e \colon X \to X$, consider the subset $U_e = \{ x \mid e(x)=x \} \subseteq X$. Then $\chi_{U_e} = e$ because $e = \rest{e}$, and so $\cO(X)$ is isomorphic to the powerset $\mathcal{P}(X)$ of $X$.
\end{example}

Note that $\rest{f}$ is a restriction idempotent for any morphism $f$~\cite[Lemma~2.1]{cockettlack:restrictioncategories}. Therefore, $e$ is a restriction idempotent if and only if $e = \rest{f}$ for some morphism $f$. Furthermore, $\mathcal{O}(A)$ is a semilattice where $e \wedge e' = e \circ e^\prime$ and the top element is the identity morphism $1_A \colon A \to A$~\cite[Section 2.1]{cockett:range}. So in particular, $e \leq e'$ if $e \circ e' = e$. In the $\S[-]$-construction, restriction idempotents correspond precisely to the subunits of the base category. 

\begin{proposition}\label{prop:restrictionidempotentsinS} 
 Let $\cat{C}$ be a firm monoidal category. The restriction idempotents in $\S[\cat{C}]$ are precisely the morphisms of the form $\big[s,\,\begin{minipic}
   \draw (0,0) to (0,-.5) node[right=-1mm,font=\tiny]{$A$};
   \draw (.3,-.2) node[dot]{} to (.3,-.5) node[right=-1mm,font=\tiny]{$S$};
  \end{minipic}\!\big]$.
 This gives a semilattice isomorphism $\cO(A) \cong \ISub(\cat{C})$.
\end{proposition}
\begin{proof} 
 As in any restriction category, $\big[s,\;\begin{minipic}
    \node[morphism] (f) at (0,0) {$f$};
    \draw ([xshift=-1mm]f.south west) to +(0,-.2) node[right=-1mm,font=\tiny]{$A$};
    \draw ([xshift=1mm]f.south east) to +(0,-.2) node[right=-1mm,font=\tiny]{$S$};
    \draw (f.north) to +(0,.15) node[right=-1mm,font=\tiny]{$A$};
   \end{minipic}\!\big] \in \cO(A)$ if and only if
   \[\Big[s,\;\begin{minipic}
    \node[morphism] (f) at (0,0) {$f$};
    \draw ([xshift=-1mm]f.south west) to +(0,-.2) node[right=-1mm,font=\tiny]{$A$};
    \draw ([xshift=1mm]f.south east) to +(0,-.2) node[right=-1mm,font=\tiny]{$S$};
    \draw (f.north) to +(0,.15) node[right=-1mm,font=\tiny]{$A$};
   \end{minipic}\!\Big] = \rest{ \Big[t,\;\begin{minipic}
    \node[morphism] (f) at (0,0) {$g$};
    \draw ([xshift=-1mm]f.south west) to +(0,-.2) node[right=-1mm,font=\tiny]{$A$};
    \draw ([xshift=1mm]f.south east) to +(0,-.2) node[right=-1mm,font=\tiny]{$T$};
    \draw (f.north) to +(0,.15) node[right=-1mm,font=\tiny]{$B$};
   \end{minipic}\!\Big]}\] for some 
  $\big[t,\,\begin{minipic}
    \node[morphism] (f) at (0,0) {$g$};
    \draw ([xshift=-1mm]f.south west) to +(0,-.125) node[right=-1mm,font=\tiny]{$A$};
    \draw ([xshift=1mm]f.south east) to +(0,-.125) node[right=-1mm,font=\tiny]{$T$};
    \draw (f.north) to +(0,.125) node[right=-1mm,font=\tiny]{$B$};
   \end{minipic}\!\big]$. 
   Thus
   $\big[s,\,\begin{minipic}
   \draw (0,0) to (0,-.4) node[right=-1mm,font=\tiny]{$A$};
   \draw (.3,-.2) node[dot]{} to (.3,-.4) node[right=-1mm,font=\tiny]{$S$};
  \end{minipic}\!\big] \mapsto s$
   is a bijection $\phi_A \colon \cO(A) \to \ISub(\cat{C})$. It is clear that $\phi_A$ preserves the top element. It remains to show that it preserves meets:
 \begin{align*}
\phi_A\Big( \big[s,\,\begin{minipic}[baseline={(0,0)}]
   \draw (0,0) to (0,-.4) node[right=-1mm,font=\tiny]{$A$};
   \draw (.3,-.2) node[dot]{} to (.3,-.4) node[right=-1mm,font=\tiny]{$S$};
  \end{minipic}\!\big] \wedge \big[t,\,\begin{minipic}[baseline={(0,0)}]
   \draw (0,0) to (0,-.4) node[right=-1mm,font=\tiny]{$A$};
   \draw (.3,-.2) node[dot]{} to (.3,-.4) node[right=-1mm,font=\tiny]{$T$};
  \end{minipic}\!\big] \Big) 
  &=~ \phi_A\Big( \big[s,\,\begin{minipic}[baseline={(0,0)}]
   \draw (0,0) to (0,-.4) node[right=-1mm,font=\tiny]{$A$};
   \draw (.3,-.2) node[dot]{} to (.3,-.4) node[right=-1mm,font=\tiny]{$S$};
  \end{minipic}\big] \circ \big[t,\,\begin{minipic}[baseline={(0,0)}]
   \draw (0,0) to (0,-.4) node[right=-1mm,font=\tiny]{$A$};
   \draw (.3,-.2) node[dot]{} to (.3,-.4) node[right=-1mm,font=\tiny]{$T$};
  \end{minipic}\!\big] \Big) \\
  &=~ \phi_A\Big( \big[s \wedge t,\,\begin{minipic}[baseline={(0,0)}]
   \draw (0,0) node[right=-1mm,font=\tiny]{$A$} to (0,.4);
   \draw (.3,0) node[right=-1mm,font=\tiny]{$S$} to (.3,.2) node[dot]{};
   \draw (.6,0) node[right=-1mm,font=\tiny]{$T$} to (.6,.2) node[dot]{};
  \end{minipic}\!\big] \Big) \\
  &=~ \big( s \wedge t \colon S \otimes T \rightarrowtail I \big) \\
  &=~ \big( s \colon S \rightarrowtail I \big) \wedge \big( t \colon T \rightarrowtail I \big) \\
  &=~ \phi_A\Big( \big[s,\,\begin{minipic}
   \draw (0,0) to (0,-.5) node[right=-1mm,font=\tiny]{$A$};
   \draw (.3,-.2) node[dot]{} to (.3,-.5) node[right=-1mm,font=\tiny]{$S$};
  \end{minipic}\!\big] \Big) \wedge \phi_A\Big( \big[t,\,\begin{minipic}
   \draw (0,0) to (0,-.5) node[right=-1mm,font=\tiny]{$A$};
   \draw (.3,-.2) node[dot]{} to (.3,-.5) node[right=-1mm,font=\tiny]{$T$};
  \end{minipic}\!\big] \Big)
 \end{align*}
 Hence $\cO(A)$ is isomorphic to $\ISub(\cat{C})$ as semilattices. 
\end{proof}

An important subclass of morphisms of a restriction category is its class of restriction-total morphisms. Intuitively, these are the morphisms which are ``totally defined''. 

\begin{definition}\cite[Section 2.1.2]{cockettlack:restrictioncategories}
 In a restriction category $\cat{X}$, a \emph{restriction-total morphism} is a morphism $f \colon A \to B$ such that $\rest{f} = \id[A]$. The subcategory of all restriction-total morphisms of $\cat\mathsf{X}$ is denoted by $\T[\cat{X}]$.
\end{definition}

\begin{example} 
 In $\cat{Par}$, the restriction-total morphism are precisely the total functions in the classical sense. So $\T[\cat{Par}] = \cat{Set}$.  
\end{example}

For $\S[\cat{C}]$, its subcategory of total morphisms is precisely the base category $\cat{C}$. 

\begin{proposition}\label{prop:totalinS}
 Let $\cat{C}$ be a firm monoidal category. A morphism in $\S[\cat{C}]$ is restric\-tion-total if and only if it is of the form $\big[1,\;\begin{minipic}
   \node[morphism] (f) at (0,0) {$f$};
   \draw (f.south) to +(0,-.125) node[right=-1mm,font=\tiny]{$A$};
   \draw (f.north) to +(0,.125) node[right=-1mm,font=\tiny]{$B$};
  \end{minipic}\!\big] $ for some morphism $f \colon A \to B$ in $\cat{C}$. This induces an isomorphism of categories $\cat{C} \cong \T\left[\S[\cat{C}]\right]$. 
\end{proposition}
\begin{proof} 
 If $\big[s,\;\begin{minipic}
    \node[morphism] (f) at (0,0) {$f$};
    \draw ([xshift=-1mm]f.south west) to +(0,-.125) node[right=-1mm,font=\tiny]{$A$};
    \draw ([xshift=1mm]f.south east) to +(0,-.125) node[right=-1mm,font=\tiny]{$S$};
    \draw (f.north) to +(0,.125) node[right=-1mm,font=\tiny]{$B$};
   \end{minipic}\!\big]$ is a restriction-total morphism, then:
 \[  
  \big[s,\,\begin{minipic}
   \draw (0,0) to (0,-.5) node[right=-1mm,font=\tiny]{$A$};
   \draw (.3,-.2) node[dot]{} to (.3,-.5) node[right=-1mm,font=\tiny]{$S$};
  \end{minipic}\!\big]  
  = 
  \rest{\big[s,\;\begin{minipic}
    \node[morphism] (f) at (0,0) {$f$};
    \draw ([xshift=-1mm]f.south west) to +(0,-.2) node[right=-1mm,font=\tiny]{$A$};
    \draw ([xshift=1mm]f.south east) to +(0,-.2) node[right=-1mm,font=\tiny]{$S$};
    \draw (f.north) to +(0,.15) node[right=-1mm,font=\tiny]{$B$};
   \end{minipic}\!\big]} 
  = [1,\begin{minipic}
   \draw (0,.4) to (0,0) node[right=-1mm,font=\tiny]{$A$};
  \end{minipic}\!]\] 
  Therefore $s = 1$, and the restriction-total morphisms are of the form claimed.
  It is straightforward that this induces an isomorphism of categories $\cat{C} \cong \T\left[\S[\cat{C}]\right]$. 
\end{proof} 

The last class of morphisms that we will study are restriction isomorphisms.

\begin{definition} In a restriction category $\cat{X}$, a \emph{restriction isomorphism} is a morphism $f \colon A \to B$ which has a map $f^\circ \colon B \to A$, called its \emph{restriction inverse}, such that $f^\circ \circ f = \rest{f}$ and $f \circ f^\circ = \rest{f^\circ}$. 
\end{definition}

If a restriction inverse exists, it is unique.

\begin{example} 
 In $\cat{Par}$, define the domain and image of $f \colon X \to Y$ as usual:
\begin{align*}
\mathsf{dom}(f) = \{ x \mid \rest{f}(x) = x \} \subseteq X && \mathsf{im}(f) = \{ y \mid \exists x \in X \colon f(x) = y \} \subseteq Y
\end{align*}
 Then $f$ is a restriction isomorphism if and only if the canonical total function $\mathsf{dom}(f) \to \mathsf{im}(f)$ given by $x \mapsto f(x)$ is a bijection.
\end{example}

Restriction isomorphisms in $\S[\cat{C}]$ correspond to isomorphisms of a certain type in $\cat{C}$.

\begin{proposition}\label{prop:restrictionisosinS}
 Let $\cat{C}$ be a firm monoidal category. A morphism $\smash{\big[s,\;\begin{minipic}
    \node[morphism] (f) at (0,0) {$f$};
    \draw ([xshift=-1mm]f.south west) to +(0,-.125) node[right=-1mm,font=\tiny]{$A$};
    \draw ([xshift=1mm]f.south east) to +(0,-.125) node[right=-1mm,font=\tiny]{$S$};
    \draw (f.north) to +(0,.125) node[right=-1mm,font=\tiny]{$B$};
   \end{minipic}\!\big]}$ in $\S[\cat{C}]$ is a restriction isomorphism if and only if 
 \begin{equation}\label{eq:restrictedisomorphism}
  \begin{pic}
  \node[morphism] (f) at (0,0) {$f$};
  \node[dot] (d) at (.4,-.4) {};
  \draw (f.north) to +(0,.2) node[above]{$B$};
  \draw ([xshift=-1mm]f.south west) to +(0,-.5) node[below]{$A$};
  \draw (d) to +(0,-.3) node[below]{$S$};
  \draw (d) to[out=180,in=-90] ([xshift=1mm]f.south east);
  \draw (d) to[out=0,in=-90] +(.25,.4) to +(.25,.8) node[above]{$S$};
 \end{pic}\end{equation}
 is an isomorphism in $\cat{C}$. 
\end{proposition}
\begin{proof}
 By definition, $\big[t,\;\begin{minipic}
    \node[morphism] (f) at (0,0) {$g$};
    \draw ([xshift=-1mm]f.south west) to +(0,-.2) node[right=-1mm,font=\tiny]{$B$};
    \draw ([xshift=1mm]f.south east) to +(0,-.2) node[right=-1mm,font=\tiny]{$T$};
    \draw (f.north) to +(0,.15) node[right=-1mm,font=\tiny]{$A$};
   \end{minipic}\!\big]$ is a restriction inverse of $\big[s,\;\begin{minipic}
    \node[morphism] (f) at (0,0) {$f$};
    \draw ([xshift=-1mm]f.south west) to +(0,-.2) node[right=-1mm,font=\tiny]{$A$};
    \draw ([xshift=1mm]f.south east) to +(0,-.2) node[right=-1mm,font=\tiny]{$S$};
    \draw (f.north) to +(0,.15) node[right=-1mm,font=\tiny]{$B$};
   \end{minipic}\!\big]$ if and only 
   \begin{align*}
\rest{\big[s,\;\begin{minipic}
    \node[morphism] (f) at (0,0) {$f$};
    \draw ([xshift=-1mm]f.south west) to +(0,-.2) node[right=-1mm,font=\tiny]{$A$};
    \draw ([xshift=1mm]f.south east) to +(0,-.2) node[right=-1mm,font=\tiny]{$S$};
    \draw (f.north) to +(0,.15) node[right=-1mm,font=\tiny]{$B$};
   \end{minipic}\!\big]} &= \big[t,\;\begin{minipic}
    \node[morphism] (f) at (0,0) {$g$};
    \draw ([xshift=-1mm]f.south west) to +(0,-.2) node[right=-1mm,font=\tiny]{$B$};
    \draw ([xshift=1mm]f.south east) to +(0,-.2) node[right=-1mm,font=\tiny]{$T$};
    \draw (f.north) to +(0,.15) node[right=-1mm,font=\tiny]{$A$};
   \end{minipic}\!\big] \circ \big[s,\;\begin{minipic}
    \node[morphism] (f) at (0,0) {$f$};
    \draw ([xshift=-1mm]f.south west) to +(0,-.2) node[right=-1mm,font=\tiny]{$A$};
    \draw ([xshift=1mm]f.south east) to +(0,-.2) node[right=-1mm,font=\tiny]{$S$};
    \draw (f.north) to +(0,.15) node[right=-1mm,font=\tiny]{$B$};
   \end{minipic}\!\big] \\
    \rest{\big[t,\;\begin{minipic}
    \node[morphism] (f) at (0,0) {$g$};
    \draw ([xshift=-1mm]f.south west) to +(0,-.2) node[right=-1mm,font=\tiny]{$B$};
    \draw ([xshift=1mm]f.south east) to +(0,-.2) node[right=-1mm,font=\tiny]{$T$};
    \draw (f.north) to +(0,.15) node[right=-1mm,font=\tiny]{$A$};
   \end{minipic}\!\big]} &= \big[s,\;\begin{minipic}
    \node[morphism] (f) at (0,0) {$f$};
    \draw ([xshift=-1mm]f.south west) to +(0,-.2) node[right=-1mm,font=\tiny]{$A$};
    \draw ([xshift=1mm]f.south east) to +(0,-.2) node[right=-1mm,font=\tiny]{$S$};
    \draw (f.north) to +(0,.15) node[right=-1mm,font=\tiny]{$B$};
   \end{minipic}\!\big] \circ \big[t,\;\begin{minipic}
    \node[morphism] (f) at (0,0) {$g$};
    \draw ([xshift=-1mm]f.south west) to +(0,-.2) node[right=-1mm,font=\tiny]{$B$};
    \draw ([xshift=1mm]f.south east) to +(0,-.2) node[right=-1mm,font=\tiny]{$T$};
    \draw (f.north) to +(0,.15) node[right=-1mm,font=\tiny]{$A$};
   \end{minipic}\!\big]
\end{align*}
In $\S[\cat{C}]$, this means $t=s$ and:
 \begin{equation}\tag{$*$}
  \begin{pic}
  \node[morphism,width=6mm] (f) at (0,0) {$f$};
  \node[morphism,width=6mm] (g) at (.6,.75) {$g$};
  \node[dot] (d) at (.5,-.5) {};
  \draw (d) to[out=180,in=-90] (f.south east);
  \draw (d) to[out=0,in=-90,looseness=.7] (g.south east);
  \draw (d) to +(0,-.3) node[below]{$S$};
  \draw (f.south west) to +(0,-.6) node[below]{$A$};
  \draw (f.north east) to node[left]{$B$} +(0,.35);
  \draw (g.north) to +(0,.2) node[above]{$A$};
 \end{pic}
 \; = \;
 \begin{pic}
  \draw (0,0) node[below]{$A$} to (0,1.9) node[above]{$A$};
  \draw (.3,0) node[below]{$S$} to (.3,1) node[dot]{};
 \end{pic}
 \qquad\qquad
 \begin{pic}
  \node[morphism,width=6mm] (f) at (0,0) {$g$};
  \node[morphism,width=6mm] (g) at (.6,.75) {$f$};
  \node[dot] (d) at (.5,-.5) {};
  \draw (d) to[out=180,in=-90] (f.south east);
  \draw (d) to[out=0,in=-90,looseness=.7] (g.south east);
  \draw (d) to +(0,-.3) node[below]{$S$};
  \draw (f.south west) to +(0,-.6) node[below]{$B$};
  \draw (f.north east) to node[left]{$A$} +(0,.35);
  \draw (g.north) to +(0,.2) node[above]{$B$};
 \end{pic}
 \; = \;
 \begin{pic}
  \draw (0,0) node[below]{$B$} to (0,1.9) node[above]{$B$};
  \draw (.3,0) node[below]{$S$} to (.3,1) node[dot]{};
 \end{pic}
 \end{equation}
 We will show that this is the case if and only if~\eqref{eq:restrictedisomorphism} is an isomorphism, with inverse:
 \begin{equation}\label{eq:restrictedinverse}
  \begin{pic}
  \node[morphism] (f) at (0,0) {$g$};
  \node[dot] (d) at (.4,-.4) {};
  \draw (f.north) to +(0,.2) node[above]{$A$};
  \draw ([xshift=-1mm]f.south west) to +(0,-.5) node[below]{$B$};
  \draw (d) to +(0,-.3) node[below]{$S$};
  \draw (d) to[out=180,in=-90] ([xshift=1mm]f.south east);
  \draw (d) to[out=0,in=-90] +(.25,.4) to +(.25,.8) node[above]{$S$};
  \end{pic}
 \end{equation}
 If~\eqref{eq:restrictedinverse} inverts~\eqref{eq:restrictedisomorphism},
 then:
 \[
  \begin{pic}
  \node[morphism,width=6mm] (f) at (0,0) {$f$};
  \node[morphism,width=6mm] (g) at (.6,.75) {$g$};
  \node[dot] (d) at (.5,-.5) {};
  \draw (d) to[out=180,in=-90] (f.south east);
  \draw (d) to[out=0,in=-90,looseness=.7] (g.south east);
  \draw (d) to +(0,-.3) node[below]{$S$};
  \draw (f.south west) to +(0,-.6) node[below]{$A$};
  \draw (f.north east) to node[left]{$B$} +(0,.35);
  \draw (g.north) to +(0,.2) node[above]{$A$};
  \end{pic}
  \; = \;
  \begin{pic}
  \node[morphism,width=6mm] (f) at (0,0) {$f$};
  \node[morphism,width=6mm] (g) at (.6,.75) {$g$};
  \node[dot] (d) at (.6,-.5) {};
  \node[dot] (D) at (1.1,.1) {};
  \draw (d) to[out=180,in=-90] (f.south east);
  \draw (d) to[out=0,in=-90,looseness=.7] (D);
  \draw (D) to[out=180,in=-90] (g.south east);
  \draw (D) to[out=0,in=-90] +(.3,.5) to +(.3,.9) node[dot]{};
  \draw (d) to +(0,-.3) node[below]{$S$};
  \draw (f.south west) to +(0,-.6) node[below]{$A$};
  \draw (f.north east) to node[left]{$B$} +(0,.35);
  \draw (g.north) to +(0,.2) node[above]{$A$};
  \end{pic}
  \; = \;  
  \begin{pic}
  \draw (0,0) node[below]{$A$} to (0,1.9) node[above]{$A$};
  \draw (.3,0) node[below]{$S$} to (.3,1) node[dot]{};
  \end{pic}
 \]
 Similarly with the composition in the other order. Therefore ($*$) holds.

 Conversely, if ($*$) holds then:
 \[
  \begin{pic}
   \node[morphism,width=6mm] (f) at (0,0) {$f$};
   \node[morphism,width=6mm] (g) at (.6,.75) {$g$};
   \node[dot] (d) at (.7,-.6) {};
   \node[dot] (D) at (1.2,0) {};
   \draw (d) to[out=180,in=-90] (f.south east);
   \draw (d) to[out=0,in=-90] (D.south);
   \draw (D) to[out=180,in=-90] (g.south east);
   \draw (d) to +(0,-.6) node[below]{$S$};
   \draw (f.south west) to +(0,-1) node[below]{$A$};
   \draw (f.north east) to node[left]{$B$} +(0,.35);
   \draw (g.north) to +(0,.2) node[above]{$A$};
   \draw (D) to[out=0,in=-90] +(.4,.5) to +(.4,1.15) node[above]{$S$};
  \end{pic}
  \;=\;
  \begin{pic}
   \node[morphism,width=6mm] (f) at (0,0) {$f$};
   \node[morphism,width=6mm] (g) at (.6,.75) {$g$};
   \node[dot] (d) at (.5,-.5) {};
   \node[dot] (D) at (.9,-.9) {};
   \draw (d) to[out=180,in=-90,looseness=.7] (f.south east);
   \draw (d) to[out=0,in=-90,looseness=.7] (g.south east);
   \draw (D) to[out=180,in=-90] (d);
   \draw (D) to +(0,-.3) node[below]{$S$};
   \draw (f.south west) to +(0,-1) node[below]{$A$};
   \draw (f.north east) to node[left]{$B$} +(0,.35);
   \draw (g.north) to +(0,.2) node[above]{$A$};
   \draw (D) to[out=0,in=-90] +(.4,.5) to +(.4,2.05) node[above]{$S$};
  \end{pic}
  \;=\;
  \begin{pic}
   \draw (0,0) node[below]{$A$} to +(0,2.3) node[above]{$A$};
   \node[dot] (d) at (.7,1) {};
   \draw (.7,0) node[below]{$S$} to (d);
   \draw (d) to[out=180,in=-90] +(-.4,.4) node[dot]{};
   \draw (d) to[out=0,in=-90] +(.4,.4) to +(.4,1.3) node[above]{$S$};
  \end{pic}
  \;=\;
  \begin{pic}
   \draw (0,0) node[below]{$A$} to +(0,2.3) node[above]{$A$};
   \draw (.5,0) node[below]{$S$} to +(0,2.3) node[above]{$S$};
  \end{pic}
 \]
 The composition in the other order is similar, showing that~\eqref{eq:restrictedisomorphism} is invertible.
\end{proof}

Recall that an \emph{inverse category} is a restriction category where every morphism has a restriction inverse~\cite[2.3.2]{cockettlack:restrictioncategories}. The previous proposition immediately gives us a characterisation of when $\S[\cat{C}]$ (and hence, jumping ahead slightly, any tensor-restriction category) is an inverse category. Later on, we will show that this gives an instance of the $\S$-construction after Corollary~\ref{cor:simple} below.

\begin{corollary}\label{cor:inversecat}
  Let $\cat{C}$ be a firm monoidal category. Then $\S[\cat{C}]$ is an inverse category if and only if $\cat{C}$ is a groupoid.
\end{corollary}
\begin{proof}
  By definition, $\S[\cat{C}]$ is an inverse category if and only if every morphism is a restriction isomorphism. By Proposition~\ref{prop:restrictionisosinS} this happens exactly when any map of the form~\eqref{eq:restrictedisomorphism} is an isomorphism in $\cat{C}$ for every subunit $s$.
  Taking $s=1$ implies that $\cat{C}$ is a groupoid.
  Conversely, if $\cat{C}$ is a groupoid, then $\ISub(\cat{C})=\{1\}$, and hence~\eqref{eq:restrictedisomorphism} is invertible.
\end{proof}

We finish this section by exhibiting an important aspect of the $\S[-]$-construction: it produces an orthogonal factorisation system, that intuitively separates the restriction aspect from the base category aspect.

\begin{proposition}
 If $\cat{C}$ is a firm monoidal category, then $\S[\cat{C}]$ has an orthogonal factorisation system given by:
 \begin{align*}
  \mathcal{E} & = \Big\{ \text{restriction isomorphisms of the form } \big[s,\;\begin{minipic}
    \node[morphism] (f) at (0,0) {$f$};
    \draw ([xshift=-1mm]f.south west) to +(0,-.2) node[right=-1mm,font=\tiny]{$A$};
    \draw ([xshift=1mm]f.south east) to +(0,-.2) node[right=-1mm,font=\tiny]{$S$};
    \draw ([xshift=-1mm]f.north west) to +(0,.2) node[right=-1mm,font=\tiny]{$B$};
       \draw ([xshift=1mm]f.north east) to +(0,.2) node[right=-1mm,font=\tiny]{$S$};
   \end{minipic}\!\big] \Big\}   \\
  \mathcal{M} & = \Big\{ \text{restriction-total morphisms } \big[1,\;\begin{minipic}
   \node[morphism] (f) at (0,0) {$f$};
   \draw (f.south) to +(0,-.2) node[right=-1mm,font=\tiny]{$A$};
   \draw (f.north) to +(0,.2) node[right=-1mm,font=\tiny]{$B$};
  \end{minipic}\!\big] \Big\}
 \end{align*}
\end{proposition}
\begin{proof}
 Clearly both $\mathcal{M}$ and $\mathcal{E}$ are closed under composition and contain all isomorphisms.
 Furthermore, any morphism $\big[s,\;\begin{minipic}
    \node[morphism] (f) at (0,0) {$f$};
    \draw ([xshift=-1mm]f.south west) to +(0,-.125) node[right=-1mm,font=\tiny]{$A$};
    \draw ([xshift=1mm]f.south east) to +(0,-.125) node[right=-1mm,font=\tiny]{$S$};
    \draw (f.north) to +(0,.125) node[right=-1mm,font=\tiny]{$B$};
   \end{minipic}\!\big]$ factors as the restriction isomorphism $[s,\begin{minipic}[baseline={(0,.2)}]
   \draw (-.3,.3) to (-.3,0) node[right=-1mm,font=\tiny]{$A$};
     \draw (0,.3) to (0,0) node[right=-1mm,font=\tiny]{$S$};
  \end{minipic}\!]\colon A \to A\otimes S$ in $\mathcal{E}$ followed by the restriction-total morphism $\big[1,\;\begin{minipic}
    \node[morphism] (f) at (0,0) {$f$};
    \draw ([xshift=-1mm]f.south west) to +(0,-.125) node[right=-1mm,font=\tiny]{$A$};
    \draw ([xshift=1mm]f.south east) to +(0,-.125) node[right=-1mm,font=\tiny]{$S$};
    \draw (f.north) to +(0,.125) node[right=-1mm,font=\tiny]{$B$};
   \end{minipic}\!\big] \colon A \otimes S \to B$ in $\mathcal{M}$. 
 \[\begin{pic}[xscale=2,yscale=1.5]
  \node (tl) at (0,1) {$A$};
  \node (tr) at (2,1) {$B$};
  \node (b) at (1,0) {$A \otimes S$};
  \draw[->] (tl) to node[above]{$\Big[s,\;\begin{minipic}
    \node[morphism] (f) at (0,0) {$f$};
    \draw[-] ([xshift=-1mm]f.south west) to +(0,-.125) node[right=-1mm,font=\tiny]{$A$};
    \draw[-] ([xshift=1mm]f.south east) to +(0,-.125) node[right=-1mm,font=\tiny]{$S$};
    \draw[-] (f.north) to +(0,.125) node[right=-1mm,font=\tiny]{$B$};
   \end{minipic}\!\Big]$} (tr);
  \draw[->,dashed] (tl) to node[left]{$\mathcal{E} \ni [s,\begin{minipic}[baseline={(0,.3)},solid]
   \draw[-] (-.3,.3) to (-.3,0) node[right=-1mm,font=\tiny]{$A$};
     \draw[-] (0,.3) to (0,0) node[right=-1mm,font=\tiny]{$S$};
  \end{minipic}\!]$} (b);
  \draw[->,dashed] (b) to node[right]{$\big[1,\;\begin{minipic}
    \node[solid,morphism] (f) at (0,0) {$f$};
    \draw[-] ([xshift=-1mm]f.south west) to +(0,-.125) node[right=-1mm,font=\tiny]{$A$};
    \draw[-] ([xshift=1mm]f.south east) to +(0,-.125) node[right=-1mm,font=\tiny]{$S$};
    \draw[-] (f.north) to +(0,.125) node[right=-1mm,font=\tiny]{$B$};
   \end{minipic}\!\big] \in \mathcal{M}$} (tr);
 \end{pic}\]
 We will show that any commuting square as below has a unique diagonal fill-in:
 \[\begin{pic}[xscale=3,yscale=1.5]
  \node (tl) at (0,1) {$A$};
  \node (tr) at (1,1) {$B \otimes S$};
  \node (bl) at (0,0) {$C$};
  \node (br) at (1,0) {$D$};
  \draw[->] (tl) to node[above]{$\big[s,\;\begin{minipic}[-]
    \node[morphism] (f) at (0,0) {$f$};
    \draw ([xshift=-1mm]f.south west) to +(0,-.125) node[right=-1mm,font=\tiny]{$A$};
    \draw ([xshift=1mm]f.south east) to +(0,-.125) node[right=-1mm,font=\tiny]{$S$};
    \draw (f.north) to +(0,.125) node[right=-1mm,font=\tiny]{$B$};
   \end{minipic}\!\big] \in \mathcal{E}$} (tr);
  \draw[->] (tr) to node[right]{$[t,h]$} (br);
  \draw[->] (tl) to node[left]{$[r,g]$} (bl);
  \draw[->] (bl) to node[below]{$[\id[I],k] \in \mathcal{M}$} (br);
  \draw[->,dashed] (tr) to node[right]{$[t,m]$} (bl);
 \end{pic}\]
 That the outer square commutes means that $r = s \wedge t$ and:
 
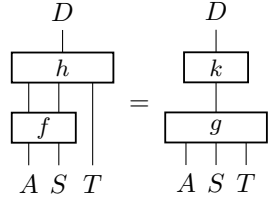
\[
  \begin{pic}
   \node[morphism,width=10mm] (h) at (0,1) {$h$};
   \node[morphism,width=5mm] (f) at (-.25,.2) {$f$};
   \draw (h.north) to +(0,.3) node[above]{$D$};
   \draw (f.south west) to +(0,-.3) node[below]{$A$};
   \draw (f.south east) to +(0,-.3) node[below]{$S$};
   \draw (f.north west) to +(0,.4);
   \draw (f.north east) to +(0,.4);
   \draw (h.south east) to +(0,-1.1) node[below]{$T$};
  \end{pic}
  \;=\;
  \begin{pic}
   \node[morphism,width=10mm] (g) at (0,0) {$g$};
   \node[morphism,width=5mm] (k) at (0,.8) {$k$};
   \draw (g.north) to +(0,.4);
   \draw (k.north) to +(0,.3) node[above]{$D$};
   \draw (g.south west) to +(0,-.3) node[below]{$A$};
   \draw (g.south) to +(0,-.3) node[below]{$S$};
   \draw (g.south east) to +(0,-.3) node[below]{$T$};
  \end{pic}
 \]
 The fact that $\Big[s,\;\begin{minipic}
    \node[morphism] (f) at (0,0) {$f$};
    \draw ([xshift=-1mm]f.south west) to +(0,-.125) node[right=-1mm,font=\tiny]{$A$};
    \draw ([xshift=1mm]f.south east) to +(0,-.125) node[right=-1mm,font=\tiny]{$S$};
    \draw ([xshift=-1mm]f.north west) to +(0,.125) node[right=-1mm,font=\tiny]{$B$};
       \draw ([xshift=1mm]f.north east) to +(0,.125) node[right=-1mm,font=\tiny]{$S$};
   \end{minipic}\!\Big] \in \mathcal{E}$ means that it has a restriction inverse $\Big[s,\;\begin{minipic}
    \node[morphism,width=8mm,anchor=north] (f) at (0,0) {$f^\circ$};
    \draw ([xshift=-1mm]f.south west) to +(0,-.125) node[right=-1mm,font=\tiny]{$B$};
    \draw ([xshift=1mm]f.south east) to +(0,-.125) node[right=-1mm,font=\tiny]{$S$};
    \draw (f.north) to +(0,.125) node[right=-1mm,font=\tiny]{$A$};
        \draw (f.south) to +(0,-.125) node[right=-1mm,font=\tiny]{$S$};
   \end{minipic}\!\Big]$:
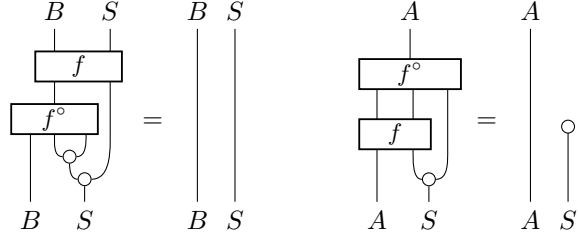
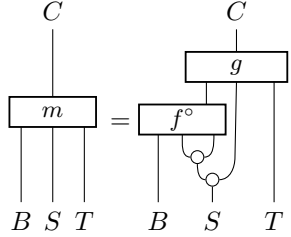
 \[
  \begin{pic}
   \node[morphism,width=8mm] (g) at (0,1) {$f$};
   \node[morphism,width=8mm,anchor=north] (f) at ([yshift=-3mm]g.south west) {$f^\circ$};
   \draw (f.north) to (g.south west);
   \draw ([xshift=1mm]g.north east) to +(0,.3) node[above]{$S$};
   \draw (g.north west) to +(0,.3) node[above]{$B$};
   \node[dot] (D) at ([xshift=2mm,yshift=-3mm]f.south) {};
   \draw (D) to[out=180,in=-90] (f.south);
   \draw (D) to[out=0,in=-90] ([xshift=1mm]f.south east);
   \node[dot] (d) at ([xshift=2mm,yshift=-3mm]D) {};
   \draw (d) to[out=180,in=-90] (D);
   \draw (d) to +(0,-.3) node[below]{$S$};
   \draw (d) to[out=0,in=-90] ([xshift=1mm,yshift=-8mm]g.south east) to ([xshift=1mm]g.south east);
   \draw (f.south west) to +(0,-.9) node[below]{$B$};
  \end{pic}
  \;=\;
  \begin{pic}
   \draw (0,0) node[below]{$B$} to (0,2.3) node[above]{$B$};
   \draw (.5,0) node[below]{$S$} to +(0,2.3) node[above]{$S$};
  \end{pic}
  \qquad\qquad
  \begin{pic}
   \node[morphism,width=10mm] (fo) at (0,1) {$f^\circ$};
   \node[morphism,width=6mm] (f) at (-.2,.2) {$f$};
   \draw (f.north west) to +(0,.4);
   \draw (f.north east) to +(0,.4);
   \node[dot] (d) at (.25,-.4) {};
   \draw (d) to[out=180,in=-90,looseness=.8] (f.south east);
   \draw (d) to[out=0,in=-90] +(.25,.4) to ([xshift=1mm]fo.south east);
   \draw (d) to +(0,-.3) node[below]{$S$};
   \draw (f.south west) to +(0,-.7) node[below]{$A$};
   \draw (fo.north) to +(0,.4) node[above]{$A$};
  \end{pic}
  \;=\;
  \begin{pic}
   \draw (0,0) node[below]{$A$} to +(0,2.3) node[above]{$A$};
   \draw (.5,0) node[below]{$S$} to +(0,1) node[dot]{};
  \end{pic}
 \]
 Define $m \colon B \otimes S \otimes T \to C$ as follows:
 \[\begin{pic}
    \node[morphism,width=8mm,anchor=north] (f) at (0,0) {$m$};
    \draw ([xshift=-1mm]f.south west) to +(0,-1) node[below]{$B$};
    \draw ([xshift=1mm]f.south east) to +(0,-1) node[below]{$T$};
    \draw (f.north) to +(0,.9) node[above]{$C$};
        \draw (f.south) to +(0,-1) node[below]{$S$};
   \end{pic} = \begin{pic}
  \node[morphism,width=10mm] (g) at (0,1) {$g$};
  \node[morphism,width=8mm,anchor=north east] (f) at ([yshift=-3mm]g.south west) {$f^\circ$};
  \draw (f.north east) to (g.south west);
  \draw (g.north) to +(0,.3) node[above]{$C$};
  \draw ([xshift=1mm]g.south east) to +(0,-1.6) node[below]{$T$};
  \node[dot] (D) at ([xshift=2mm,yshift=-3mm]f.south) {};
  \draw (D) to[out=180,in=-90] (f.south);
  \draw (D) to[out=0,in=-90] ([xshift=1mm]f.south east);
  \node[dot] (d) at ([xshift=2mm,yshift=-3mm]D) {};
  \draw (d) to[out=180,in=-90] (D);
  \draw (d) to +(0,-.3) node[below]{$S$};
  \draw (d) to[out=0,in=-90] +(.3,.3) to (g.south);
  \draw (f.south west) to +(0,-.9) node[below]{$B$};
 \end{pic}\]
 Then both triangles commute by Lemma~\ref{lem:SsissS}:
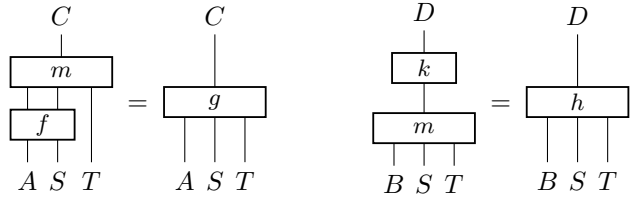
 \[ 
  \begin{pic}
   \node[morphism,width=10mm] (h) at (0,1) {$m$};
   \node[morphism,width=5mm] (f) at (-.25,.3) {$f$};
   \draw (h.north) to +(0,.3) node[above]{$C$};
   \draw (f.south west) to +(0,-.3) node[below]{$A$};
   \draw (f.south east) to +(0,-.3) node[below]{$S$};
   \draw (f.north west) to +(0,.3);
   \draw (f.north east) to +(0,.3);
   \draw (h.south east) to +(0,-1) node[below]{$T$};
  \end{pic}
  \;=\;
  \begin{pic}
   \node[morphism,width=10mm] (g) at (0,0) {$g$};
   \draw (g.north) to +(0,.7) node[above]{$C$};
   \draw (g.south west) to +(0,-.6) node[below]{$A$};
   \draw (g.south) to +(0,-.6) node[below]{$S$};
   \draw (g.south east) to +(0,-.6) node[below]{$T$};
  \end{pic}
  \qquad\qquad
  \begin{pic}
   \node[morphism,width=10mm] (g) at (0,0) {$m$};
   \node[morphism,width=5mm] (k) at (0,.8) {$k$};
   \draw (g.north) to +(0,.4);
   \draw (k.north) to +(0,.3) node[above]{$D$};
   \draw (g.south west) to +(0,-.3) node[below]{$B$};
   \draw (g.south) to +(0,-.3) node[below]{$S$};
   \draw (g.south east) to +(0,-.3) node[below]{$T$};
  \end{pic}
  \;=\;
  \begin{pic}
   \node[morphism,width=10mm] (g) at (0,0) {$h$};
   \draw (g.north) to +(0,.7) node[above]{$D$};
   \draw (g.south west) to +(0,-.6) node[below]{$B$};
   \draw (g.south) to +(0,-.6) node[below]{$S$};
   \draw (g.south east) to +(0,-.6) node[below]{$T$};
  \end{pic}
 \]
 Clearly $\Big[t,\;\begin{minipic}
    \node[morphism,width=8mm,anchor=north] (f) at (0,0) {$m$};
    \draw ([xshift=-1mm]f.south west) to +(0,-.125) node[right=-1mm,font=\tiny]{$B$};
    \draw ([xshift=1mm]f.south east) to +(0,-.125) node[right=-1mm,font=\tiny]{$T$};
    \draw (f.north) to +(0,.125) node[right=-1mm,font=\tiny]{$C$};
        \draw (f.south) to +(0,-.125) node[right=-1mm,font=\tiny]{$S$};
   \end{minipic}\!\Big]$ is the unique map achieving this. 
\end{proof}

\section{Monoidal Restriction Categories}\label{sec:monoidalrestrictioncats}

In this section, we will show that $\S[-]$-construction results in a firm monoidal category whose subunits are precisely those of the base category. Furthermore, we will also explain how this monoidal structure is compatible with the restriction structure, which we call a monoidal restriction category. 

\begin{proposition}\label{prop:monoidalinS}
 If $\cat{C}$ is a firm monoidal category, then $\S[\cat{C}]$ is a firm monoidal category with monoidal unit $I$ and the monoidal product $\otimes$ defined on objects as $A \otimes B$ and on morphisms as follows: 
 \[
  \Big[s,\;\begin{minipic}
    \node[morphism] (f) at (0,0) {$f$};
    \draw ([xshift=-1mm]f.south west) to +(0,-.125) node[right=-1mm,font=\tiny]{$A$};
    \draw ([xshift=1mm]f.south east) to +(0,-.125) node[right=-1mm,font=\tiny]{$S$};
    \draw (f.north) to +(0,.125) node[right=-1mm,font=\tiny]{$B$};
   \end{minipic}\!\Big] \otimes \Big[t,\;\begin{minipic}
    \node[morphism] (f) at (0,0) {$g$};
    \draw ([xshift=-1mm]f.south west) to +(0,-.125) node[right=-1mm,font=\tiny]{$C$};
    \draw ([xshift=1mm]f.south east) to +(0,-.125) node[right=-1mm,font=\tiny]{$T$};
    \draw (f.north) to +(0,.125) node[right=-1mm,font=\tiny]{$D$};
   \end{minipic}\!\Big] = 
  \bigg[s \wedge t,\;\;\begin{minipic}
   \node[morphism,width=5mm] (f) at (0,0) {$f$};
   \node[morphism,width=5mm] (g) at (1,0) {$g$};
   \draw (f.south west) to +(0,-.3) node[right=-1mm,font=\tiny]{$A$};
   \draw (g.south west) to[out=-90,in=90] ([yshift=-3mm]f.south east) node[right=-1mm,font=\tiny]{$C$};
   \draw[     preaction={draw,white,line width=2pt,-},
     preaction={draw,white,ultra thick, shorten >=-2\pgflinewidth}
] (f.south east) to[out=-90,in=90] ([yshift=-3mm]g.south west) node[right=-1mm,font=\tiny] {$S$};
   \draw (g.south east) to +(0,-.3) node[right=-1mm,font=\tiny] {$T$};
   \draw (f.north) to +(0,.125) node[right=-1mm,font=\tiny]{$B$};
   \draw (g.north) to +(0,.125) node[right=-1mm,font=\tiny]{$D$};
  \end{minipic}\!\bigg]
 \]
 Subunits in $\S[\cat{C}]$ are exactly the maps $\big[1,\,\begin{minipic}
   \draw (.3,-.1) node[dot]{} to (.3,-.3) node[right=-1mm,font=\tiny]{$S$};
  \end{minipic}\!\big] $ for subunits $s\colon S \rightarrowtail I$ in $\cat{C}$. 
 Hence $\ISub(\cat{C}) \cong \ISub(\S[\cat{C}])$ as semilattices. 
\end{proposition}
\begin{proof}
 The coherence isomorphisms $\alpha$, $\rho$, $\lambda$, and $\sigma$ of $\cat{C}$ induce coherence isomorphisms $[\id[I],\alpha]$, $[\id[I], \rho]$, $[\id[I], \lambda]$, and $[\id[I], \sigma]$ for $\S[\cat{C}]$, in such a way that the triangle, pentagon, and hexagon equations are satisfied. In particular, we note that the interchange law comes down to:
 \[
  \left[s \wedge s' \wedge t \wedge t',\;\;\begin{minipic}
   \node[morphism,width=4mm] (f) at (0,0) {$f$};
   \node[morphism,width=4mm,anchor=south west] (f') at ([xshift=1mm,yshift=1.5mm]f.north) {$f'$};
   \draw (f.north) to ([xshift=-1mm]f'.south west);
   \draw (f.south west) to +(0,-.5);
   \draw (f'.north) to +(0,.15);
   \node[morphism,width=4mm] (g) at (1.2,0) {$g$};
   \node[morphism,width=4mm,anchor=south west] (g') at ([xshift=1mm,yshift=1.5mm]g.north) {$g'$};
   \draw (g.north) to ([xshift=-1mm]g'.south west);
   \draw (g.south west) to[out=-90,in=90] +(-.9,-.5);
   \draw (g.south east) to +(0,-.5) node[right=-1mm,font=\tiny]{$T$};
   \draw ([xshift=1mm]g'.south east) to +(0,-.95) node[right=-1mm,font=\tiny]{$T'$};
   \draw (g'.north) to +(0,.15);
   \draw[halo] (f.south east) to[out=-90,in=90] +(.4,-.5) node[right=-1mm,font=\tiny]{$S$};
   \draw[halo] ([xshift=1mm]f'.south east) to +(0,-.5) to[out=-90,in=90] +(.4,-.95) node[right=-1mm,font=\tiny]{$S'$};
  \end{minipic}\!\!\right]
  \;=\; 
  \left[s \wedge t \wedge s' \wedge t',\;\;\begin{minipic}
   \node[morphism,width=4mm] (f) at (0,0) {$f$};
   \node[morphism,width=4mm,anchor=south west] (f') at ([xshift=1mm,yshift=3mm]f.north) {$f'$};
   \draw (f.north) to ([xshift=-1mm]f'.south west);
   \draw (f.south west) to +(0,-.3);
   \draw (f'.north) to +(0,.15);
   \node[morphism,width=4mm] (g) at (.9,0) {$g$};
   \node[morphism,width=4mm,anchor=south west] (g') at ([xshift=1mm,yshift=3mm]g.north) {$g'$};
   \draw (g.north) to ([xshift=-1mm]g'.south west);
   \draw (g.south west) to[out=-90,in=90] +(-.55,-.3);
   \draw (g.south east) to +(0,-.3) node[right=-1mm,font=\tiny]{$T$};
   \draw ([xshift=1mm]g'.south east) to[out=-90,in=90,looseness=.8] +(.4,-.4) to +(.4,-.9) node[right=-1mm,font=\tiny]{$T'$};
   \draw (g'.north) to +(0,.15);
   \draw[halo] (f'.south east) to[out=-90,in=90,looseness=.5] +(1,-.3) to +(1,-.9) node[right=-1mm,font=\tiny]{$S'$};
   \draw[halo] (f.south east) to[out=-90,in=90] +(.55,-.3) node[right=-1mm,font=\tiny]{$S$};
  \end{minipic}\!\!\right]
 \]
 It is clear that any subunit $s$ in $\cat{C}$ induces a subunit $\big[1,\,\begin{minipic}
   \draw (.3,-.1) node[dot]{} to (.3,-.3) node[right=-1mm,font=\tiny]{$S$};
  \end{minipic}\!\big]$ in $\S[\cat{C}]$. Conversely, because monomorphisms are total \cite[Lemma 2.2.i]{cockettlack:restrictioncategories}, a subunit in $\S[\cat{C}]$ is of the form $\big[1,\;\begin{minipic}[baseline={(0,.05)}]
    \node[morphism] (f) at (0,0) {$f$};
    \draw (f.south) to +(0,-.125) node[right=-1mm,font=\tiny]{$A$};
   \end{minipic}\big]$. Since by Proposition \ref{prop:totalinS}, we have the isomorphism $\cat{C} \cong \T\left[\S[\cat{C}]\right]$, it follows that $f \circ \rho^{-1}$ must also be monic in $\cat{C}$, and that $f \otimes A \otimes I$ is invertible in $\cat{C}$. Therefore $f \circ \rho^{-1}$ represents a subunit $A \to I$ in $\cat{C}$. So we conclude that $\ISub(\cat{C}) \cong \ISub(\S[\cat{C}])$. Finally, if $\big[1,\,\begin{minipic}
   \draw (.3,-.1) node[dot]{} to (.3,-.3) node[right=-1mm,font=\tiny]{$S$};
  \end{minipic}\!\big]$ and $\big[1,\,\begin{minipic}
   \draw (.3,-.1) node[dot]{} to (.3,-.3) node[right=-1mm,font=\tiny]{$T$};
  \end{minipic}\!\big]$ are subunits in $\S[\cat{C}]$, then by construction so is their tensor product. Hence $\S[\cat{C}]$ is firm. 
\end{proof}

The monoidal structure of $\S[\cat{C}]$ is compatible with the restriction structure.

\begin{definition}
 A \emph{monoidal restriction category} 
 is a monoidal category $\cat{X}$ that is also a restriction category where:
 \[\rest{f \otimes g} = \rest{f} \otimes \rest{g}\] 
 A \emph{firm restriction category} is a monoidal restriction category whose underlying monoidal category is firm. 
\end{definition}

Because isomorphisms are always total~\cite[Lemma 2.2.i]{cockettlack:restrictioncategories}, the coherence isomorphisms in a monoidal restriction category are always total. 

\begin{example} \normalfont 
  Every cartesian restriction category~\cite[Definition~2.16]{cockettcruttwellgallagher:differentialrestriction} is a monoidal restriction category, just as every cartesian category is a monoidal category. So in particular, $\cat{Par}$ is a monoidal restriction category with monoidal structure given by the cartesian product of sets, and so $\rest{f \times g} = \rest{f} \times \rest{g}$. Furthermore, $\cat{Par}$ is a firm monoidal restriction category where the subunits are (up to isomorphism) the singleton $\{ \ast \}$ and the empty set $\emptyset$.  
\end{example} 

\begin{corollary} 
 If $\cat{C}$ is a firm monoidal category, then $\S[\cat{C}]$ is a firm restriction category. Furthermore, the isomorphism $\cat{C} \simeq \T[\S[\cat{C}]]$ of Proposition~\ref{prop:totalinS} is an isomorphism of firm monoidal categories. 
\end{corollary}
\begin{proof} It is clear from the monoidal product defined in Propostion~\ref{prop:monoidalinS} that:
\[\rest{\Big[s,\;\begin{minipic}
    \node[morphism] (f) at (0,0) {$f$};
    \draw ([xshift=-1mm]f.south west) to +(0,-.2) node[right=-1mm,font=\tiny]{$A$};
    \draw ([xshift=1mm]f.south east) to +(0,-.2) node[right=-1mm,font=\tiny]{$S$};
    \draw (f.north) to +(0,.15) node[right=-1mm,font=\tiny]{$B$};
   \end{minipic}\!\Big] \otimes \Big[t,\;\begin{minipic}
    \node[morphism] (f) at (0,0) {$g$};
    \draw ([xshift=-1mm]f.south west) to +(0,-.2) node[right=-1mm,font=\tiny]{$C$};
    \draw ([xshift=1mm]f.south east) to +(0,-.2) node[right=-1mm,font=\tiny]{$T$};
    \draw (f.north) to +(0,.15) node[right=-1mm,font=\tiny]{$D$};
   \end{minipic}\!\Big]} = \rest{\Big[s,\;\begin{minipic}
    \node[morphism] (f) at (0,0) {$f$};
    \draw ([xshift=-1mm]f.south west) to +(0,-.2) node[right=-1mm,font=\tiny]{$A$};
    \draw ([xshift=1mm]f.south east) to +(0,-.2) node[right=-1mm,font=\tiny]{$S$};
    \draw (f.north) to +(0,.15) node[right=-1mm,font=\tiny]{$B$};
   \end{minipic}\!\Big]} \otimes \rest{ \Big[t,\;\begin{minipic}
    \node[morphism] (f) at (0,0) {$g$};
    \draw ([xshift=-1mm]f.south west) to +(0,-.2) node[right=-1mm,font=\tiny]{$C$};
    \draw ([xshift=1mm]f.south east) to +(0,-.2) node[right=-1mm,font=\tiny]{$T$};
    \draw (f.north) to +(0,.15) node[right=-1mm,font=\tiny]{$D$};
   \end{minipic}\!\Big]}\]
 So $\S[\cat{C}]$ is a firm restriction category. It is clear that the isomorphism $\cat{C} \to \T[\S[\cat{C}]]$ of Propostion~\ref{prop:monoidalinS} is (strong) monoidal and preserves subunits. 
\end{proof}

We conclude this section by observing that having dual objects or being closed in the $\S[-]$-construction is closely related to the same properties of the base category, and vice-versa. We denote dual objects in a monoidal category $\cat{C}$ as $A \dashv A^*$ (see~\cite[Chapter 3]{heunenvicary:cqm} for the full definition), and if $\cat{C}$ is closed, we write $\multimap$ for internal homs.

\begin{lemma}
 $A \dashv A^*$ is a duality in $\cat{C}$ if and only if $A \dashv A^*$ is a duality in $\S[\cat{C}]$.
\end{lemma}
\begin{proof}
 The functor $\cat{C} \to \S[\cat{C}]$ given by $A \mapsto A$ on objects and $f \mapsto [1,f]$ on morphisms is (strong) monoidal, and monoidal functors preserve dual objects~\cite[Theorem~3.14]{heunenvicary:cqm}. So if $A \dashv A^*$ in $\cat{C}$, then also $A \dashv A^*$ in $\S[\cat{C}]$. Conversely, suppose that $A \dashv A^*$ in $\S[\cat{C}]$. Then there are $[s,\eta] \colon I \to A^* \otimes A$ and $[t,\varepsilon] \colon A \otimes A^* \to I$ in $\S[\cat{C}]$ satisfying:
 \[
  [\id,A] 
  = \Big[t, \,\begin{minipic}
    \node[morphism,width=6mm] (eta) at (0,0) {$\varepsilon$};
    \draw (eta.south) to +(0,-.2) node[right=-1mm]{$A^*$};
    \draw ([xshift=-1mm]eta.south west) to +(0,-.2) node[right=-1mm]{$A$};
    \draw ([xshift=1mm]eta.south east) to[out=-90,in=90] +(.4,-.2) node[right=-1mm]{$T$};
    \draw (.35,-.35) node[right=-1mm]{$A$} to[out=90,in=-90] (.75,-.1) to (.75,.3) node[right=-1mm]{$A$};
   \end{minipic}\Big] 
   \circ 
   \Big[s, \,\begin{minipic}
    \node[morphism,width=4mm] (eta) at (0,0) {$\eta$};
    \draw ([xshift=-.5mm]eta.north west) to +(0,.2) node[right=-1mm]{$A^*$};
    \draw ([xshift=.5mm]eta.north east) to +(0,.2) node[right=-1mm]{$A$};
    \draw (eta.south) to +(0,-.2) node[right=-1mm]{$S$};
    \draw (-.6,-.35) node[right=-1mm]{$A$} to (-.6,.35) node[right=-1mm]{$A$};
    \end{minipic}\Big]
  = \Bigg[s\wedge t, \begin{minipic}
   \node[morphism,width=10mm] (epsilon) at (0,1) {$\varepsilon$};
   \node[morphism,width=5mm,anchor=north west] (eta) at ([yshift=-4mm]epsilon.south) {$\eta$};
   \draw (epsilon.south west) to +(0,-.9) node[right=-1mm]{$A$};
   \draw (eta.north west) to node[right=-1mm]{$A^*$} (epsilon.south);
   \draw (eta.south) to +(0,-.2) node[right=-1mm]{$S$};
   \draw (eta.north east) to[out=90,in=-90] +(.5,.5) to +(.5,.8) node[right=-1mm]{$A$};
   \draw (epsilon.south east) to[out=-90,in=90] +(.5,-.5) to +(.5,-.9) node[right=-1mm]{$T$};
  \end{minipic}\Bigg]
 \]
 But then $s=t=1$, and $\eta$ and $\varepsilon$ witness $A \dashv A^*$ in $\cat{C}$.
\end{proof}

\begin{lemma}
 A firm monoidal category $\cat{C}$ is closed if and only if $\S[\cat{C}]$ is closed.
\end{lemma}
\begin{proof}
 Suppose $\cat{C}$ is closed.
 If $\varepsilon \colon (A \multimap B) \otimes A \to B$ is the counit
 of the adjunction $(-) \otimes A \dashv A \multimap (-)$ in $\cat{C}$, then 
 \begin{align*}
   \S[\cat{C}](A, B \multimap C) & \to \S[\cat{C}](A \otimes B,C) \\
   \Big[s,\;\begin{minipic}
    \node[morphism] (f) at (0,0) {$f$};
    \draw ([xshift=-1mm]f.south west) to +(0,-.125) node[right=-1mm,font=\tiny]{$A$};
    \draw ([xshift=1mm]f.south east) to +(0,-.125) node[right=-1mm,font=\tiny]{$S$};
    \draw (f.north) to +(0,.125) node[right=-1mm,font=\tiny]{$B$};
   \end{minipic}\!\Big] & \mapsto 
   \Bigg[s,\,\begin{minipic}
    \node[morphism,width=4mm] (epsilon) at (0,.5) {$\varepsilon$};
    \node[morphism,width=4mm,anchor=north] (f) at ([xshift=-1mm,yshift=-3mm]epsilon.south west) {$f$};
    \draw (f.north) to node[left=-1mm]{$B \multimap C$} ([xshift=-1mm]epsilon.south west);
    \draw (epsilon.north) to +(0,.125) node[right=-1mm]{$C$};
    \draw ([xshift=-1mm]f.south west) to +(0,-.3) node[right=-1mm]{$A$};
    \draw ([xshift=1mm]f.south east) to[out=-90,in=90] +(.3,-.3) node[right=-1mm]{$S$};
    \draw ([xshift=1mm]epsilon.south east) to +(0,-.5) to[out=-90,in=90] +(-.3,-.9) node[right=-1mm]{$B$};
   \end{minipic}\Bigg]
 \end{align*}
 is a well-defined natural bijection. Hence $\S[\cat{C}]$ is closed. Conversely, suppose that $\S[\cat{C}]$ is closed. Then there are morphisms:
 \begin{align*}
[s,\eta_B] \colon B \to A \multimap (B \otimes A) && [t,\varepsilon_B] \colon (A \multimap B) \otimes A \to B
\end{align*}
satisfying: $[t, \varepsilon_{A \otimes B}] \circ [s, \eta_B \otimes A] = [\id[I], A \otimes B]$. But then $s \wedge t = 1$ so $s=t=1$, and so $\eta$ and $\varepsilon$ are the unit and counit of an adjunction $(-) \otimes A \dashv A \multimap (-)$ in $\cat{C}$.
\end{proof}

\section{Tensor-restriction categories}\label{sec:tensorrestrictioncats}

The goal of this section is to axiomatise the restriction categories of the form $\S[\cat{C}]$ for some firm monoidal category $\cat{C}$. We call such restriction categories \emph{tensor-restriction categories}. Explicitly, we will show that $\cat{X} \simeq \S[\T[\cat{X}]]$ for a tensor-restriction category $\cat{X}$. In fact, the category of firm monoidal categories is equivalent to the category of tensor-restriction categories.

A key concept for this section is the notion of restriction in the sense of~\cite[Section 4]{enriquemolinerheunentull:tensortopology}, which can be defined in arbitrary monoidal categories.

\begin{definition}\label{def:restricttosubunit}\cite[Definition 4.1]{enriquemolinerheunentull:tensortopology}
 In a monoidal category $\cat{C}$, a morphism ${f \colon A \to B}$ \emph{tensor-restricts} to a subunit $s \colon S \rightarrowtail I$ if it factors via $B \otimes s$ in the sense that the following diagram commutes: 
 \[\begin{pic}
  \node (tl) at (0,1) {$A$};
  \node (tr) at (3,1) {$B$};
  \node (bl) at (0,0) {$B \otimes S$};
  \node (br) at (3,0) {$B \otimes I$};
  \draw[->] (tl) to node[above]{$f$} (tr);
  \draw[->,dashed] (tl) to (bl);
  \draw[->] (bl) to node[below]{$B \otimes s$} (br);
  \draw[->] (br) to node[right]{$\rho$} (tr);
 \end{pic}\]  
\end{definition}
If a morphism $f \colon A \to B$ restricts to $s$, then the map $A \to B \otimes S$ is not necessarily unique. A special case is when identity morphisms restrict to subunits.

\begin{proposition}\label{prop:identityrestrict} \cite[Proposition 4.2]{enriquemolinerheunentull:tensortopology}
 An identity morphism $\id[A] \colon A \to A$ in a monoidal category restricts to a subunit $s \colon S \rightarrowtail I$ if and only if $\rho \circ (A \otimes s)$ is an isomorphism $A \otimes S \simeq A$. 
\end{proposition} 

It follows from that, if the identity morphism on $A$ restricts to a subunit $s$, then any morphism $f \colon A \to B$ also restricts to $s$. The converse says that $f$ is totally defined in the tensor topology sense.

\begin{definition}\label{def:tensortotal}
 A morphism $f \colon A \to B$ in a monoidal category is \emph{tensor-total} when the identity morphism on $A$ tensor-restricts to a subunit $s$ as soon as $f$ tensor-restricts to $s$.
\end{definition}
 
Equivalently, $f \colon A \to B$ is tensor-total when $f$ and $A$ have the same support (for any support datum)~\cite[Section 6]{enriquemolinerheunentull:tensortopology}. Furthermore, it follows from Proposition~\ref{prop:identityrestrict} that if $f \colon A \to B$ is tensor-total and tensor-restricts to a subunit $s \colon S \rightarrowtail I$, then $A \otimes S \simeq A$. The $\S[-]$-construction has a similar property.

\begin{lemma}\label{lem:tensortotalinS}
 Let $\cat{C}$ be a firm monoidal category. If $\smash{\big[s,\;\begin{minipic}
    \node[morphism] (f) at (0,0) {$f$};
    \draw ([xshift=-1mm]f.south west) to +(0,-.125) node[right=-1mm,font=\tiny]{$A$};
    \draw ([xshift=1mm]f.south east) to +(0,-.125) node[right=-1mm,font=\tiny]{$S$};
    \draw (f.north) to +(0,.125) node[right=-1mm,font=\tiny]{$B$};
   \end{minipic}\!\big]}$ is tensor-total in $\S[\cat{C}]$, then $A \simeq A \otimes S$ in $\cat{C}$.
\end{lemma}
\begin{proof}
 Any morphism $\smash{\big[s,\;\begin{minipic}
    \node[morphism] (f) at (0,0) {$f$};
    \draw ([xshift=-1mm]f.south west) to +(0,-.125) node[right=-1mm,font=\tiny]{$A$};
    \draw ([xshift=1mm]f.south east) to +(0,-.125) node[right=-1mm,font=\tiny]{$S$};
    \draw (f.north) to +(0,.125) node[right=-1mm,font=\tiny]{$B$};
   \end{minipic}\!\big]}$ in $\S[\cat{C}]$ restricts to the subunit $\smash{\big[1,\,\begin{minipic}
   \draw (.3,-.2) node[dot]{} to (.3,-.4) node[right=-1mm,font=\tiny]{$S$};
  \end{minipic}\!\big]}$:
 \[\begin{pic}
  \node (l) at (0,1.5) {$A$};
  \node (r) at (6,1.5) {$B$};
  \node (b) at (3,0) {$B \otimes S$};
  \draw[->] (l) to node[above]{$\Big[s,\;\begin{minipic}[-]
    \node[morphism] (f) at (0,0) {$f$};
    \draw ([xshift=-1mm]f.south west) to +(0,-.125) node[right=-1mm,font=\tiny]{$A$};
    \draw ([xshift=1mm]f.south east) to +(0,-.125) node[right=-1mm,font=\tiny]{$S$};
    \draw (f.north) to +(0,.125) node[right=-1mm,font=\tiny]{$B$};
   \end{minipic}\!\Big]$} (r);
  \draw[->] (l) to node[yshift=-3mm,left=5mm]{$\Bigg[s,\,\begin{minipic}[-]
   \node[morphism] (f) at (0,.3) {$f$};
   \node[dot] (d) at (.1,0) {};
   \draw (f.south west) to +(0,-.3) node[right=-1mm]{$A$};
   \draw (f.south east) to[out=-90,in=180] (d);
   \draw (d) to +(0,-.15) node[right=-1mm]{$S$};
   \draw (f.north) to +(0,.2) node[right=-1mm]{$B$};
   \draw (d) to[out=0,in=-90] +(.2,.2) to +(.2,.65) node[right=-1mm]{$S$};
  \end{minipic}\Bigg]$} (b);
  \draw[->] (b) to node[right=5mm]{$\Big[\id[I],\,\begin{minipic}[-,baseline={(0,.3)}]
   \draw (0,-.2) node[right=-1mm]{$B$} to +(0,.5);
   \node[dot] (d) at (.25,.1) {};
   \draw (d) to +(0,-.3) node[right=-1mm]{$S$};
  \end{minipic}\Big]$} (r);
 \end{pic}\]
 So if $\big[s,\;\begin{minipic}
    \node[morphism] (f) at (0,0) {$f$};
    \draw ([xshift=-1mm]f.south west) to +(0,-.125) node[right=-1mm,font=\tiny]{$A$};
    \draw ([xshift=1mm]f.south east) to +(0,-.125) node[right=-1mm,font=\tiny]{$S$};
    \draw (f.north) to +(0,.125) node[right=-1mm,font=\tiny]{$B$};
   \end{minipic}\!\big]$ is tensor-total, then $A$ restricts to $\big[\id[I],\,\begin{minipic}
   \draw (.3,-.2) node[dot]{} to (.3,-.4) node[right=-1mm,font=\tiny]{$S$};
  \end{minipic}\!\big]$ too. 
  Therefore $A \otimes S \cong A$ in $\S[\cat{C}]$ by Proposition~\ref{prop:identityrestrict}. However, since isomorphisms are total, and because $\cat{C} \simeq \T[\S[\cat{C}]]$ by Proposition~\ref{prop:monoidalinS}, it follows that $A \simeq A \otimes S$ in $\cat{C}$.
\end{proof}



The next key ingredient in characterising the $\S[-]$-construction is points (in the sense of morphisms $I \to X$) that are restriction isomorphisms whose restriction inverse is tensor-total.

\begin{definition}
 A \emph{tensor-restriction point} in a firm restriction category $\cat{X}$ is a morphism $d \colon I \to X$ with a restriction inverse $d^\circ\colon X \to I$ such that $d^\circ$ is tensor-total.
\end{definition}

In the $\S[-]$-construction, tensor-restriction points are characterised by the subunits of the base category.

\begin{lemma}\label{lem:pointsinSC}
 Let $\cat{C}$ be a firm category. The tensor-restriction points in $\S[\cat{C}]$ are precisely the morphisms of the form $[s,\begin{minipic}
   \draw (0,0) to (0,-.4) node[right=-1mm,font=\tiny]{$S$};
  \end{minipic}\!] \colon I \to S$ for a subunit $s \colon S \rightarrowtail I$ in $\cat{C}$.
\end{lemma}
\begin{proof}
 We will prove that tensor-total restriction isomorphisms into $I$ are precisely the morphisms in $\S[\cat{C}]$ of the form $\big[s,\,\begin{minipic}
   \draw (0,-.1) node[dot]{} to (0,-.4) node[right=-1mm,font=\tiny]{$S$};
     \draw (.3,-.1) node[dot]{} to (.3,-.4) node[right=-1mm,font=\tiny]{$S$};
  \end{minipic}\!\big]$. Because this has $[s, \,\begin{minipic} \draw (0,-.2) node[right=-1mm]{$S$} to (0,.1); \end{minipic}\!] \colon I \to S$ as restriction inverse, and restriction inverses are unique, the claim then follows. 

  It is clear that $\big[s,\,\begin{minipic}
   \draw (0,-.1) node[dot]{} to (0,-.4) node[right=-1mm,font=\tiny]{$S$};
     \draw (.3,-.1) node[dot]{} to (.3,-.4) node[right=-1mm,font=\tiny]{$S$};
  \end{minipic}\!\big] \colon S \to I$ is a restriction isomorphism from Proposition~\ref{prop:restrictionisosinS}. To see that it is tensor-total: if it restricts to $\big[\id[I],\,\begin{minipic}
   \draw (.3,-.1) node[dot]{} to (.3,-.4) node[right=-1mm,font=\tiny]{$T$};
  \end{minipic}\!\big]$, then $s \wedge s = t \circ m$ for some $m \colon S \otimes S \to T$, whence $s = t \circ m \circ (S \otimes s)^{-1}$, and so $s \leq t$.

 For the converse, suppose that $\big[s,\;\begin{minipic}
    \node[morphism] (f) at (0,0) {$f$};
    \draw ([xshift=-1mm]f.south west) to +(0,-.125) node[right=-1mm,font=\tiny]{$A$};
    \draw ([xshift=1mm]f.south east) to +(0,-.125) node[right=-1mm,font=\tiny]{$S$};
    \draw (f.north) to +(0,.125) node[right=-1mm,font=\tiny]{$B$};
   \end{minipic}\!\big] \colon A \to I$ is a tensor-total restriction isomorphism. Then $A \simeq A \otimes S$ by Lemma~\ref{lem:tensortotalinS}, and $A \otimes S \simeq S$ by Proposition~\ref{prop:restrictionisosinS}, giving $A \simeq S$ by the following isomorphism $m \colon A \to S$:
 \[\begin{pic}
  \node[morphism,width=6mm] (f) at (0,1) {$f$};
  \node[dot] (d) at (.5,.5) {};
  \node[morphism] (i) at (.25,0) {$(A \otimes s)^{-1}$};
  \draw (f.south east) to[out=-90,in=180] (d);
  \draw (d) to +(0,-.3);
  \draw (f.south west) to node[left]{$A$} +(0,-.6);
  \draw (d) to[out=0,in=-90] +(.3,.3) to +(.3,1) node[right]{$S$};
  \draw (i.south) to +(0,-.4) node[left]{$A$};
 \end{pic}\]
 Define $a = s \circ m \colon A \rightarrowtail I$.
 Lemma~\ref{lem:SsissS} now gives:
 \[f \otimes f = \begin{pic}
  \node[morphism,width=14mm] (t) at (0,.6) {$f$};
  \node[morphism,width=3mm] (b) at (.1,0) {$f$};
  \node[dot] (d) at (.45,-.5) {};
  \draw (d) to[out=180,in=-90,looseness=.9] (b.south east);
  \draw (d) to[out=0,in=-90,looseness=.5] ([xshift=2mm]t.south east);
  \draw (d) to +(0,-.4) node[right=-1mm]{$S$};
  \draw (b.south west) to +(0,-.7) node[right=-1mm]{$A$};
  \draw ([xshift=-2mm]t.south west) to +(0,-1.3) node[right=-1mm]{$A$};
  \draw (-.4,-.9) node[right=-1mm]{$S$} to +(0,.4) node[dot]{};
 \end{pic}\]
 It follows that $\big[s,\;\begin{minipic}
    \node[morphism] (f) at (0,0) {$f$};
    \draw ([xshift=-1mm]f.south west) to +(0,-.125) node[right=-1mm,font=\tiny]{$A$};
    \draw ([xshift=1mm]f.south east) to +(0,-.125) node[right=-1mm,font=\tiny]{$S$};
    \draw (f.north) to +(0,.125) node[right=-1mm,font=\tiny]{$B$};
   \end{minipic}\!\big] = [a, f \circ (A \otimes m)] = [a, a \wedge a]$.
\end{proof}

The last preparation we need before we can define a tensor-restriction category is to recall the notion of scalar multiplication in a monoidal category \cite[Section 2.1.3]{heunenvicary:cqm}. In a monoidal category $\cat{C}$, given a scalar $s \colon I \to I$ and a morphism $f \colon A \to B$, define the (left) scalar multiplication $s \bullet f \colon A \to B$ as the composite $s \bullet f = \lambda_B \circ (s \otimes f) \circ \lambda^{-1}_B$. 

Now we can axiomatise the types of restriction categories we are interested in, and show that they characterise the $\S[-]$-construction.

\begin{definition}\label{def:tensorrestrictioncategory}
  A \emph{tensor-restriction category} is a firm restriction category where: 
  \begin{enumerate}
    \item[(TR1)] 
    any restriction idempotent scalar $e=\rest{e} \colon I \to I$ factors through a subunit $s \colon S \rightarrowtail I$ via a tensor-restriction point $d \colon I \to S$:
      \[\begin{pic}[xscale=2]
        \node (tl) at (0,2) {$I$};
        \node (tr) at (2,2) {$I$};
        \node (t) at (1,1.2) {$S$};
        \draw[->] (tl) to node[above]{$e=\rest{e}$} (tr);
        \draw[->,dashed] (tl) to node[below]{$d$} (t);
        \draw[>->,dashed] (t) to node[below]{$s$} (tr);
      \end{pic}\]

  \item[(TR2)] 
  any subunit $s \colon S \rightarrowtail I$ has a tensor-restriction point $d \colon I \to S$ as restriction section in the sense that the following diagram commutes: 
      \[\begin{pic}[xscale=2]
        \node (l) at (0,0) {$I$};
        \node (r) at (2,0) {$I$};
        \node (t) at (1,1) {$S$};
        \draw[>->] (t) to node[above]{$s$} (r);
        \draw[->,halo] (l) to node[below]{$\rest{d}$} (r);
        \draw[->,dashed] (l) to node[above]{$d$} (t);
      \end{pic}\]
  \item[(TR3)] 
  any restriction idempotent $f=\rest{f} \colon X \to X$ equals $f=e \bullet X$ for a unique restriction idempotent scalar $e = \rest{e} \colon I \to I$; 
  
  \item[(TR4)]
  any tensor-total morphism $f \colon X \to Y$ equals $f=\total{f} \circ \rest{f}$ for a unique restriction-total morphism $\total{f} \colon X \to Y$; 

  \item[(TR5)]
  tensor-restriction points $d \colon I \to X$ have the left-lifting property against subunits $s \colon S \rightarrowtail I$: if $s \circ f = g \circ d$ then $f = m \circ d$ and $g = s \circ m$;
      \begin{equation}\label{eq:diagonalfillin}
       \begin{pic}[xscale=2,yscale=1.5]
        \node (tl) at (0,1) {$I$};
        \node (tr) at (1,1) {$X$};
        \node (bl) at (0,0) {$S$};
        \node (br) at (1,0) {$I$};
        \draw[->] (tl) to node[above]{$d$} (tr);
        \draw[->] (tr) to node[right]{$g$} (br);
        \draw[->] (tl) to node[left]{$f$} (bl);
        \draw[>->] (bl) to node[below]{$s$} (br);
        \draw[->,dashed] (tr) to node[right=2mm]{$m$} (bl);
       \end{pic}
      \end{equation}   

  \item[(TR6)]
  if $d \colon I \to X$ and $d' \colon I \to X'$ are tensor-restriction points, then so is their tensor product $(d \otimes d') \circ \lambda^{-1} \colon I \to X \otimes X'$;  
 
  \item[(TR7)]
  tensor-restriction points are determined by their codomain: if $d,d' \colon I \to X$ are tensor-restriction points, then $d'=d \circ m$ for a unique scalar $m \colon I \to I$.
 \end{enumerate}
\end{definition}

The seven axioms (TR1)--(TR7) essentially demand that a notion from restriction category theory agrees with the corresponding notion from tensor topology.

\begin{proposition}
  If $\cat{C}$ is a firm monoidal category, $\S[\cat{C}]$ is a tensor-restriction category.
\end{proposition}
\begin{proof} 
  The previous section already showed that $\S[\cat{C}]$ is a firm restriction category. It remains to verify (TR1)--(TR7). 

  For (TR1), recall that by Proposition \ref{prop:restrictionidempotentsinS}, restriction idempotent scalars in $\S[\cat{C}]$ are of the form $[s,\,\begin{minipic} \draw (0,-.2) to (0,0) node[dot]{}; \end{minipic}\,]$ for a subunit $s \colon S \rightarrowtail I$ in $\cat{C}$. 
  But this morphism factors through the subunit $[\id[I],s]$ as follows:
  \[\begin{pic}[xscale=2]
    \node (l) at (0,1) {$I$};
    \node (r) at (2,1) {$I$};
    \node (t) at (1,0) {$S$};
    \draw[->] (l) to node[above]{$[s,\,\begin{minipic}[-]
      \draw (0,-.4) node[right=-1mm]{$S$} to (0,-.2) node[dot]{};
    \end{minipic}\!]$} (r);
    \draw[->] (l) to node[below left]{$[s,\,\begin{minipic}[-]
      \draw (0,-.4) node[right=-1mm]{$S$} to (0,-.1);
    \end{minipic}\!]$} (t);
    \draw[->] (t) to node[below right]{$[\id[I],\,\begin{minipic}[-]
      \node[dot] (d) at (0,-.2) {};
      \draw (d) to +(0,-.2) node[right=-1mm]{$S$};
    \end{minipic}\!]$} (r);
  \end{pic}\]
  Axiom (TR2) holds similarly. 
  
  For (TR3), recall by Proposition \ref{prop:restrictionidempotentsinS} that restriction idempotents on $A$ in $\S[\cat{C}]$ are of the form $\smash{\big[s,\,\begin{minipic} \draw (-.3,0) node[right=-1mm,font=\tiny]{$A$} to (-.3,.4); \draw (0,0) node[right=-1mm,font=\tiny]{$S$} to (0,.2) node[dot,scale=.8]{}; \end{minipic}\!\big]}$ for a subunit $s$ in $\cat{C}$. But this is precisely the scalar multiplication of the identity on $A$ with the restriction idempotent scalar $[s,\,\begin{minipic} \draw (0,-.2) to (0,0) node[dot,scale=.8]{}; \end{minipic}\,]$.

  We turn to (TR4). If $\big[s,\;\begin{minipic}
        \node[morphism] (f) at (0,0) {$f$};
        \draw ([xshift=-1mm]f.south west) to +(0,-.125) node[right=-1mm,font=\tiny]{$A$};
        \draw ([xshift=1mm]f.south east) to +(0,-.125) node[right=-1mm,font=\tiny]{$S$};
        \draw (f.north) to +(0,.125) node[right=-1mm,font=\tiny]{$B$};
      \end{minipic}\!\big]$ is tensor-total in $\S[\cat{C}]$, then  
   $[\id[I],\smash{\begin{minipic}[baseline={(0,.1)}] \draw (0,-.2) node[right=-1mm,font=\tiny]{$A$} to (0,.2); \draw (.3,-.2) node[right=-1mm,font=\tiny]{$S$} to (.3,0) node[minidot]{};\end{minipic}}\!]$ has an inverse $\big[\id[I],\;\begin{minipic}
        \node[morphism] (f) at (0,0) {$g$};
        \draw ([xshift=-1mm]f.north west) to +(0,.125) node[right=-1mm,font=\tiny]{$A$};
        \draw ([xshift=1mm]f.north east) to +(0,.125) node[right=-1mm,font=\tiny]{$S$};
        \draw (f.south) to +(0,-.125) node[right=-1mm,font=\tiny]{$A$};
      \end{minipic}\!\big]$ by Lemma~\ref{lem:tensortotalinS} and~\cite[Proposition 4.2]{enriquemolinerheunentull:tensortopology}. Define:
       \[\Bigl\lceil\big[s,\;\begin{minipic}
        \node[morphism] (f) at (0,0) {$f$};
        \draw ([xshift=-1mm]f.south west) to +(0,-.125) node[right=-1mm,font=\tiny]{$A$};
        \draw ([xshift=1mm]f.south east) to +(0,-.125) node[right=-1mm,font=\tiny]{$S$};
        \draw (f.north) to +(0,.125) node[right=-1mm,font=\tiny]{$B$};
      \end{minipic}\!\big]\Bigr\rceil = [\id[I],f \circ g] \colon A \to B\]
      Then $\big[s,\;\begin{minipic}
        \node[morphism] (f) at (0,0) {$f$};
        \draw ([xshift=-1mm]f.south west) to +(0,-.125) node[right=-1mm,font=\tiny]{$A$};
        \draw ([xshift=1mm]f.south east) to +(0,-.125) node[right=-1mm,font=\tiny]{$S$};
        \draw (f.north) to +(0,.125) node[right=-1mm,font=\tiny]{$B$};
      \end{minipic}\!\big] = \Bigl\lceil\big[s,\;\begin{minipic}
        \node[morphism] (f) at (0,0) {$f$};
        \draw ([xshift=-1mm]f.south west) to +(0,-.125) node[right=-1mm,font=\tiny]{$A$};
        \draw ([xshift=1mm]f.south east) to +(0,-.125) node[right=-1mm,font=\tiny]{$S$};
        \draw (f.north) to +(0,.125) node[right=-1mm,font=\tiny]{$B$};
      \end{minipic}\!\big]\Bigr\rceil \circ \rest{\big[s,\;\begin{minipic}
        \node[morphism] (f) at (0,0) {$f$};
        \draw ([xshift=-1mm]f.south west) to +(0,-.125) node[right=-1mm,font=\tiny]{$A$};
        \draw ([xshift=1mm]f.south east) to +(0,-.125) node[right=-1mm,font=\tiny]{$S$};
        \draw (f.north) to +(0,.125) node[right=-1mm,font=\tiny]{$B$};
      \end{minipic}\!\big]}$, and $\Bigl\lceil\big[s,\;\begin{minipic}
        \node[morphism] (f) at (0,0) {$f$};
        \draw ([xshift=-1mm]f.south west) to +(0,-.125) node[right=-1mm,font=\tiny]{$A$};
        \draw ([xshift=1mm]f.south east) to +(0,-.125) node[right=-1mm,font=\tiny]{$S$};
        \draw (f.north) to +(0,.125) node[right=-1mm,font=\tiny]{$B$};
      \end{minipic}\!\big]\Bigr\rceil$ is the unique such map: if $\big[s,\;\begin{minipic}
        \node[morphism] (f) at (0,0) {$f$};
        \draw ([xshift=-1mm]f.south west) to +(0,-.125) node[right=-1mm,font=\tiny]{$A$};
        \draw ([xshift=1mm]f.south east) to +(0,-.125) node[right=-1mm,font=\tiny]{$S$};
        \draw (f.north) to +(0,.125) node[right=-1mm,font=\tiny]{$B$};
      \end{minipic}\!\big] = [\id[I],h] \circ \rest{\big[s,\;\begin{minipic}
        \node[morphism] (f) at (0,0) {$f$};
        \draw ([xshift=-1mm]f.south west) to +(0,-.125) node[right=-1mm,font=\tiny]{$A$};
        \draw ([xshift=1mm]f.south east) to +(0,-.125) node[right=-1mm,font=\tiny]{$S$};
        \draw (f.north) to +(0,.125) node[right=-1mm,font=\tiny]{$B$};
      \end{minipic}\!\big]}$, then $f = h \otimes s$, and so $h=f \circ g$.

  To see (TR5), notice that~\eqref{eq:diagonalfillin} in $\S[\cat{C}]$ becomes:
  \[\begin{pic}[xscale=3,yscale=1.5]
    \node (tl) at (0,1) {$I$};
    \node (tr) at (1,1) {$R$};
    \node (bl) at (0,0) {$S$};
    \node (br) at (1,0) {$I$};
    \draw[->] (tl) to node[above]{$[r,\id[R]]$} (tr);
    \draw[->] (tr) to node[right]{$ \big[t,\;\begin{minipic}[-]
        \node[morphism] (f) at (0,0) {$g$};
        \draw ([xshift=-1mm]f.south west) to +(0,-.125) node[right=-1mm,font=\tiny]{$B$};
        \draw ([xshift=1mm]f.south east) to +(0,-.125) node[right=-1mm,font=\tiny]{$T$};
        \draw (f.north) to +(0,.125) node[right=-1mm,font=\tiny]{$C$};
      \end{minipic}\!\big]$} (br);
    \draw[->] (tl) to node[left]{$[r \wedge t,f]$} (bl);
    \draw[>->] (bl) to node[below]{$[\id[I],s]$} (br);
    \draw[->,dashed] (tr) to node[right=2mm]{$[t,m]$} (bl);
  \end{pic}\]
  There is indeed a unique diagonal fill-in, and it is $m=f$.

  Finally, (TR6) and (TR7) follow from Corollary~\ref{lem:pointsinSC}.
\end{proof}

We will now show that every tensor-restriction category $\cat{X}$ is in fact of the form $\S[\cat{C}]$ for some firm monoidal category $\cat{C}$. To explain this, let us investigate some basic properties of tensor-restriction categories, specifically regarding its restriction idempotents. Since $\ISub(\cat{C}) \cong \ISub(\S[\cat{C}])$ in $\S[\cat{C}]$ and $\cO(A) \cong \ISub(\cat{C})$ for every object $A$, we have $\cO(A) \cong \cO(I) \cong \ISub(\S[\cat{C}])$. This holds generally: if $\cat{X}$ is a tensor restriction category then $\cO(X) \cong \cO(I) \cong \ISub(\cat{X})$. To set the scene, we first consider scalars.

\begin{lemma}\label{lemma:scalarrestriction1}
  If $\cat{X}$ is a monoidal restriction category,
  then the semilattice $\cO(I)$ is a retract of the scalar monoid $\cat{X}(I,I)$.      
  \[\begin{tikzpicture}
    \node (l) at (0,0) {$\cO(I)$};
    \node (r) at (3,0) {$\cat{X}(I,I)$};
    \draw[>->] ([yshift=1mm]l.east) to ([yshift=1mm]r.west);
    \draw[->>] ([yshift=-1mm]r.west) to node[below]{$\rest{(-)}$} ([yshift=-1mm]l.east);
    \draw[->] (l) to[out=160,in=-160,looseness=3] node[left]{$1$} (l);
  \end{tikzpicture}\]
\end{lemma}
\begin{proof}
  Recall that the scalars $\cat{X}(I,I)$ of a monoidal category are always a commutative monoid under composition $s \circ t$, which equals scalar multiplication $s \bullet t = \lambda \circ (s \otimes t) \circ \lambda^{-1}$; see~\cite[Section 2.1]{heunenvicary:cqm}. Now:
  \begin{align*}
    \rest{s} \circ \rest{t}
    & = \rest{s} \bullet \rest{t} && \text{(by \cite[Lemma 2.6(b)]{heunenvicary:cqm})} \\
    & = \lambda \circ \rest{(s \otimes t)} \circ \lambda^{-1} && \text{(because $\rest{f \otimes g}=\rest{f} \otimes \rest{g}$)} \\
    & = \lambda \circ \rest{\rest{\lambda} \circ (s \otimes t)} \circ \lambda^{-1} && \text{($\lambda$ is total)}\\
    & = \lambda \circ \rest{\lambda \circ (s \otimes t)} \circ \lambda^{-1} && \text{(by \cite[Lemma 2.1(iii)]{cockettlack:restrictioncategories})} \\
    & = \lambda \circ \lambda^{-1} \circ \rest{\lambda \circ (s \otimes t)\circ \lambda^{-1}} && \text{(R4)} \\
    & = \rest{\lambda \circ (s \otimes t) \circ \lambda^{-1}} && \text{($\lambda$ is iso)} \\
    & = \rest{s \circ t}  && \text{(by \cite[Lemma 2.6(b)]{heunenvicary:cqm})}
  \end{align*}
  It follows that restriction is a monoid homomorphism $\cat{X}(I,I) \to \cO(I)$. 
  Finally, if $e \in \cO(I)$, then by definition $\rest{e}=e$, making $\cO(I)$ a retract of $\cat{X}(I,I)$, where we regard a semilattice as a commutative idempotent monoid.
\end{proof}

\begin{lemma} \label{lemma:scalarrestriction2}
  If $\cat{X}$ is a monoidal restriction category, 
  then there is a semilattice morphism $(-) \bullet X \colon \cO(I) \to \cO(X)$ for any object $X$.
\end{lemma}
\begin{proof}
  It is a monoid homomorphism by~\cite[Lemma~2.6]{heunenvicary:cqm}.
\end{proof}

Next we use (TR1) to construct a semilattice morphism $\cO(I) \to \ISub(\cat{X})$. 

\begin{lemma}\label{lem:OtoISub}
  If $\cat{X}$ is a firm restriction category satisfying (TR1), (TR5), and (TR6), then there is a semilattice morphism $\cO(I) \to \ISub(\cat{X})$.
\end{lemma}
\begin{proof}
  If $\rest{e}=e\colon I \to I$, axiom (TR1) provides a least subunit $s$ through which $e$ factors. Say $e=s \circ d$; because $s$ is monic, $d$ is unique. If $s$ and $s'$ are both least subunits through which $e$ factors, then $s \leq s'$ and $s' \leq s$ so $s'=s$. Thus also $s$ is unique, and we may call these morphisms $s_e$ and $d_e$. Hence $e \mapsto s_e$ is a well-defined function $\cO(I) \to \ISub(\cat{X})$. This function preserves top elements, $s_{\id} = 1$.
  Indeed, $\id[I]$ is the least subunit through which $\id[I]$ factors: if $\id[I]=s' \circ d'$, then $\id[I] \leq s'$.
  \[\begin{pic}[xscale=2,yscale=.8]
    \node (tl) at (0,2) {$I$};
    \node (tr) at (2,2) {$I$};
    \node (t) at (1,1.2) {$I$};
    \node (b) at (1,0) {$S'$};
    \draw[->] (tl) to node[above]{$\id$} (tr);
    \draw[->] (tl) to node[below]{$\id$} (t);
    \draw[>->] (t) to node[below]{$1$} (tr);
    \draw[->,dashed] (t) to node[right]{$d'$} (b);
    \draw[->] (tl) to[out=-90,in=180,looseness=.8] node[below=1mm]{$d'$} (b);
    \draw[>->] (b) to[out=0,in=-90,looseness=.8] node[below=1mm]{$s'$} (tr);
  \end{pic}\]
To show that the function also preserves meets, first notice that $d_e \otimes d_f \colon I \otimes I \to S_e \otimes S_f$ is a restriction isomorphism, with restriction inverse $(e \circ s_e) \otimes (f \circ s_f)$.
  Now if $t \circ m = e \circ s_e$ for a subunit $t$ and some $m \colon S_e\to T$, then $t \circ m \circ d_e = e \circ s_e \circ d_e = e \circ e = e$ and so $s_e \leq t$ by (TR1). Therefore $e \circ s_e$ is tensor-total.
  Similarly $f \circ s_f$ is tensor-total.
  By (TR6) the tensor product $d_e \otimes d_f$ is a tensor-restriction point.
  \[\begin{pic}[xscale=4,yscale=2]
        \node (tl) at (0,1) {$I$};
        \node (t) at (.3,1) {$I\otimes I$};
        \node (tr) at (1,1) {$S_e \otimes S_f$};
        \node (bl) at (0,0) {$S_{e \wedge f}$};
        \node (br) at (1,0) {$I$};
        \draw[->] (tl) to node[above]{$\simeq$} (t);
        \draw[->] (t) to node[above]{$d_e \otimes d_f$} (tr);
        \draw[>->] (tr) to node[right]{$s_e \wedge s_f$} (br);
        \draw[->] (tl) to node[left]{$d_{e\wedge f}$} (bl);
        \draw[>->] (bl) to node[below]{$s_{e \wedge f}$} (br);
        \draw[->,dashed] (tr) to (bl);
  \end{pic}\]
  Now apply (TR5) to find $s_e \wedge s_f = s_{e \wedge f}$. So we conclude that we have a semilattice morphism $\cO(I) \to \ISub(\cat{X})$.
\end{proof}

Using (TR2) we can construct a semilattice morphism $\ISub(\cat{X}) \to \cO(I)$ in the other direction.

\begin{lemma}\label{lem:ISubtoO}
  If $\cat{X}$ is a firm restriction category satisfying (TR2), (TR6), and (TR7), then there is a semilattice morphism $\ISub(\cat{X}) \to \cO(I)$.
\end{lemma}
\begin{proof}
  If $s$ is a subunit, axiom (TR2) provides a tensor-restriction point $d \colon I \to S$ such that $s \circ d = \rest{d} = \rest{s \circ d}$. 
  By (TR7) this morphism $d$ is unique up to scalar, and we may call it $d_s$. Thus $s \mapsto e_s = s \circ d_s$ is a well-defined function $\ISub(\cat{X}) \to \cO(I)$. This function preserves top elements, $e_{\id} = \id[I]$, as $\id[I]$ is a tensor-restriction-point as well as the unique restriction-section of $\id[I]$: if $\id[I] \circ d' = \overline{d'}$, then $d'=\id[I] \circ \rest{d'}$.
This function also preserves meets: if $r$ and $s$ are subunits, then 
  $d = (d_r \otimes d_s) \circ \lambda^{-1} \colon I \to R \otimes S$ is a tensor-restriction point by (TR6), and is a restriction-section of $r \wedge s$ because $(r \wedge s) \circ d = (r \circ d_r) \circ (s \circ d_s) = \rest{d_r} \bullet \rest{d_s} = \rest{d}$. We conclude there is a semilattice morphism $\ISub(\cat{X}) \to \cO(I)$.
  Hence $e_r \wedge e_s = e_{r \wedge s}$.
\end{proof}

In fact, in the presence of (TR5), the semilattice morphisms in the previous two lemmas are each other's inverse.

\begin{proposition}\label{prop:OisISub}
  If $\cat{X}$ is a firm restriction category satisfying (TR1), (TR2), (TR5), (TR6), and (TR7), then there is a semilattice isomorphism $\cO(I) \simeq \ISub(\cat{X})$.
\end{proposition}
\begin{proof}
  We will prove that the functions of Lemmas~\ref{lem:OtoISub} and~\ref{lem:ISubtoO} are inverses. Let $s$ be a subunit, and $\rest{e}=e\colon I \to I$ be a restriction idempotent scalar.
  
  On the one hand, (TR1) and (TR5) guarantee that $s_{e_s} = s$.
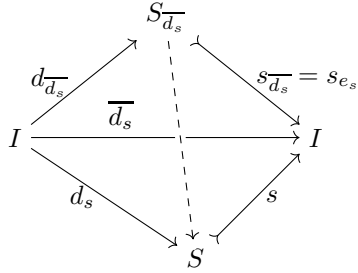
  \[\begin{pic}[xscale=2,yscale=1.6]
    \node (l) at (0,0) {$I$};
    \node (r) at (2,0) {$I$};
    \node (t) at (1,1) {$S_{\rest{d_s}}$};
    \node (b) at (1.2,-1) {$S$};
    \draw[>->] (b) to node[right]{$s$} (r);
    \draw[->] (l) to node[left]{$d_s$} (b);
    \draw[->] (l) to node[left=1mm]{$d_{\rest{d_s}}$} (t);
    \draw[>->] (t) to node[right]{$s_{\rest{d_s}} = s_{e_s}$} (r);
    \draw[->] (l) to node[left=3mm,yshift=3mm]{$\rest{d_s}$} (r);
    \draw[->,dashed,halo] (t) to (b);
  \end{pic}\]
  The dashed morphism $m$ exists by (TR5) and is unique because $s$ is monic.
  But there is also a morphism $n$ in the opposite direction. 
  Because $m \circ n$ is unique by (TR5) it has to be the identity, and similarly for $n \circ m$. Hence $m$ is an isomorphism, and $s_{e_s}=s$.

  On the other hand, $d_e$ and $d_{s_e}$ have the same codomain.
  \[\begin{pic}[xscale=2,yscale=1.6]
    \node (l) at (0,0) {$I$};
    \node (r) at (2,0) {$I$};
    \node (t) at (1,1) {$I$};
    \node (b) at (1.2,-1) {$S_e$};
    \draw[>->] (b) to node[right]{$s_e$} (r);
    \draw[->] (l) to node[left]{$d_e$} (b);
    \draw[>->] (t) to node[right=1mm]{$\rest{d_{s_e}} = e_{s_e}$} (r);
    \draw[->] (l) to node[left=3mm,yshift=2mm]{$e=\rest{e}$} (r);
    \draw[->,halo] (t) to node[right=1mm,yshift=-5mm]{$d_{s_e}$} (b);
    \draw[->,dashed] (l) to (t);
  \end{pic}\]
  Hence the dashed morphism exists by (TR7), and similarly there is a unique morphism in the opposite direction, showing that $e_{s_e}=e$. 
\end{proof}

It follows from (TR3) that $\cO(I) \simeq \ISub(\cat{X})$ for any object $X$ in a tensor-restriction category.

From now on, for a restriction idempotent scalar $e \colon I \to I$, write $s_e$ for the least subunit through which $e$ factors, and $d_e$ for the (unique) mediating map.
\[\begin{pic}
  \node (l) at (0,1) {$I$};
  \node (r) at (3,1) {$I$};
  \node (b) at (1.5,0) {$S_e$};
  \draw[->] (l) to node[above]{$e=\rest{e}$} (r);
  \draw[->,dashed] (l) to node[left=1mm]{$d_e$} (b);
  \draw[>->,dashed] (b) to node[right=1mm]{$s_e$} (r);
\end{pic}\]
By (TR3) and Lemma \ref{lemma:scalarrestriction2}, the notion of restriction from the restriction category agrees with the notion of restriction from the firm monoidal category.
Write $e_f$ for the restriction idempotent scalar satisfying $\rest{f} = e_f \bullet \mathrm{dom}(f)$, and $s_f$ and $d_f$ for the corresponding subunit and mediating map as above.
Combining (TR1) and (TR3) then factors any morphism as follows:
\[\begin{pic}[xscale=4,yscale=1.3]
  \node (tl) at (0,2) {$X$};
  \node (l) at (0,1) {$X \otimes I$};
  \node (bl) at (0,0) {$X \otimes S_f$};
  \node (tr) at (1,2) {$Y$};
  \node (r) at (1,1) {$Y \otimes I$};
  \node (br) at (1,0) {$Y \otimes S_f$};
  \draw[->] (tl) to node[above]{$f$} (tr);
  \draw[->] (l) to node[above]{$f \otimes e_f$} (r);
  \draw[->] (bl) to node[below]{$f \otimes S_f$} (br);
  \draw[->] (tl) to node[left]{$\lambda^{-1}$} (l);
  \draw[->] (r) to node[right]{$\lambda$} (tr);
  \draw[->] (l) to node[left]{$X \otimes d_f$} (bl);
  \draw[->] (br) to node[right]{$Y \otimes s_f$} (r);
\end{pic}\]

(TR4) says that the notion of totality from the restriction category agrees with the notion of totality from the firm monoidal category as in Definition~\ref{def:tensortotal}, and that we may replace the bottom map in the previous diagram by a total one.

\begin{proposition}\label{prop:factorisation}
  Any morphism $f \colon X \to Y$ in a tensor-restriction category factors via the restriction-total morphism $T(f) = \rho\circ (Y \otimes s_f) \circ \total{f \otimes S_f} \colon X \otimes S_f \to Y$.
\end{proposition}
\begin{proof}
  Suppose that $f \otimes S_f \colon X \otimes S_f \to Y \otimes S_f$ restricts to a subunit $r$, say via $g \colon X \otimes S_f \to Y \otimes S_f \otimes R$.
  Then:
  \begin{align*}
    \rest{f} 
    & = \rest{\rho \circ (Y \otimes (s_f \wedge r)) \circ g \circ (X \otimes d_f) \circ \rho^{-1}} \\
    & = \rest{(Y \otimes (s_f \wedge r)) \circ g \circ (X \otimes d_f) \circ \rho^{-1}} &&\text{(by~\cite[Lem. 2.1.iii,vi]{cockettlack:restrictioncategories})} \\ 
    & = \rest{\rest{(Y \otimes (s_f \wedge r))} \circ g \circ (X \otimes d_f \circ \rho^{-1})} &&\text{(by~\cite[Lem. 2.1.iii]{cockettlack:restrictioncategories})} \\
    & = \rest{(\rest{Y} \otimes \rest{(s_f \wedge r)}) \circ g \circ (X \otimes d_f \circ \rho^{-1})} \\    
    & = \rest{g \circ (X \otimes d_f) \circ \rho^{-1}} &&\text{(by~\cite[Lem. 2.1.vi]{cockettlack:restrictioncategories})} \\
  \intertext{Therefore:}
    \rest{g} \circ (X \otimes d_f) \circ \rho^{-1} 
    & = (X \otimes d_f) \circ \rho^{-1} \circ \rest{g \circ (X \otimes d_f) \circ \rho^{-1}} &&\text{(by R4)} \\
    & = (X \otimes d_f)\circ \rho^{-1} \circ \rest{f} 
  \end{align*}
  Because $\rest{g} = e_g \bullet (X \otimes S_f)$, now $(X \otimes s_f) \circ \rest{g} \circ (X \otimes d_f) = e_g \bullet (X \otimes s_f) \circ (X \otimes d_f)$. It follows that:
  \[e_f = e_f \bullet e_g = s_g \circ d_g \circ s_f \circ d_f = (s_f \wedge s_g) \circ (\id \otimes d_g) \circ \rho^{-1} \circ d_f\]
  But $s_f$ is the least subunit through which $e_f$ factors, so $s_f \leq s_f \wedge s_g$.
  Because $\rest{g} = \rest{(X \otimes S_f \otimes r) \circ g}$ by firmness, similarly $e_f = e_r \bullet e_g \bullet e_f$, so $s_f \leq s_g \wedge r$.
  Hence $X \otimes S_f \otimes R \simeq X \otimes S_f$, so that $X \otimes S_f$ restricts to $r$~\cite[Proposition 4.2]{enriquemolinerheunentull:tensortopology}.
  So $f \otimes S_f$ is tensor-total, and $f$ factors as:
  \[\begin{pic}[xscale=2,yscale=1.5]
    \node (Tl) at (0,2.5) {$X$};
    \node (Tr) at (3,2.5) {$Y$};
    \node (tl) at (0,1.75) {$X \otimes I$};
    \node (tr) at (3,1.75) {$Y \otimes I$};
    \node (l) at (0,1) {$X \otimes S_f$};
    \node (r) at (3,1) {$Y \otimes S_f$};
    \node (br) at (3,0) {$X \otimes S_f$};
    \draw[->] (Tl) to node[above]{$f$} (Tr);
    \draw[->] (Tl) to node[left]{$\rho^{-1}$} (tl);
    \draw[->] (tr) to node[right]{$\rho$} (Tr);
    \draw[->] (tl) to node[above]{$f \otimes e_f$} (tr);
    \draw[->] (tl) to node[left]{$X \otimes d_f$} (l);
    \draw[->] (l) to node[above]{$f \otimes S_f$} (r);
    \draw[->] (r) to node[right]{$Y \otimes s_f$} (tr);
    \draw[->] (l) to node[below left]{$\rest{f} \otimes S_f$} (br);
    \draw[->] (br) to node[right]{$\total{f \otimes S_f}$} (r);
  \end{pic}\]
  As the right vertical morphisms are restriction-total, so is their composition.
\end{proof}




\begin{lemma}
  If $\cat{X}$ is a tensor-restriction category, then $\T[\cat{X}]$ is a firm monoidal category.
\end{lemma}
\begin{proof}
  First, $\T[\cat{X}]$ is a well-defined braided monoidal category: if $f$ and $g$ are restriction-total, then $\rest{f \otimes g}=\rest{f \otimes g}$ makes $f \otimes g$ restriction-total too, and all (coherence) isomorphisms are restriction-total.

  Let $s \colon S \to I$ be a subunit in $\T[\cat{X}]$. Then $S \otimes s \colon S \otimes S \to S \otimes I$ is invertible in $\T[\cat{X}]$ and hence in $\cat{X}$. 
  Suppose that $f,g \colon A \to S$ in $\cat{X}$ satisfy $s \circ f = g \circ s$.
  Then $\rest{f} = \rest{s \circ f} = \rest{s \circ g} = \rest{g}$, and so $d_f = d_g$.
  Observe that $s$ is tensor-total in $\cat{X}$, because if $s = t \circ m$ for a subunit $t$ in $\cat{X}$, then $S \otimes t \colon S \otimes T \to S$ is an isomorphism with inverse $(S \otimes m) \circ (S \otimes s)^{-1}$.
  Now $s = s \circ \id[S] = s \circ \rest{s}$, so $\total{s} = s$ in $\cat{X}$.
  Therefore $T(s) = (\id \otimes s) \circ \total{s \otimes S} = s \wedge s \colon S \otimes S \to I$.
  It follows from (TR4) that $T(s) \circ T(f) = T(s \circ f) = T(s \circ g) = T(s) \circ T(g)$ and so $T(f)=T(g)$.
  Thus $f = T(f) \circ (\rest{f} \otimes d_f) = T(g) \circ (\rest{g} \otimes d_g) = g$.
  So $s$ is monic and hence a subunit in $\cat{X}$, too.
  Thus $\T[\cat{X}]$ is firm because $\cat{X}$ is firm.
\end{proof}

We now state the main result of this section.

\begin{theorem}\label{thm:main}
  If $\cat{X}$ is a tensor-restriction category, then there is a firm monoidal restriction category isomorphism $\cat{X} \simeq \S[\T[\cat{X}]]$.
\end{theorem}
\begin{proof}
  Define a functor $F \colon \cat{X} \to \S[\T[\cat{X}]]$ by $F(f) = [s_f, T(f)]$, which is well-defined by Proposition~\ref{prop:factorisation}. 
  Because $\total{\id[X]}=\id[X]$ and the least subunit $s_{\id[X]}$ which $\id[I]$ factors through is $\id[I]$, we have $T(\id[X])= \rho$ and $F$ indeed preserves identities.
  To see that $F$ preserves composition, let $f \colon X \to Y$ and $g \colon Y \to Z$ in $\cat{X}$.
  Then the following diagram commutes by Lemmas~\ref{lem:OtoISub} and~\ref{lem:ISubtoO}:
  \[\begin{pic}[xscale=4.5,yscale=1.5,font=\tiny]
    \node (Tl) at (0,2.7) {$X$};
    \node (tl) at (0,2) {$X \otimes I \otimes I$};
    \node (bl) at (0,1) {$X \otimes S_f \otimes I$};
    \node (Bl2) at (0,.1) {$X \otimes S_f \otimes S_g$};
    \node (Bl) at (0,-.5) {$X \otimes S_{gf}$};
    \node (T) at (1,2.7) {$Y$};
    \node (t) at (1,2) {$Y \otimes I \otimes I$};
    \node (b1) at (.7,1) {$Y \otimes S_f \otimes I$};
    \node (b2) at (1.3,1) {$Y \otimes I \otimes S_g$};
    \node (B) at (1,-.5) {$Y \otimes S_{gf}$};
    \node (B2) at (1,.1) {$Y \otimes S_f \otimes S_g$};
    \node (Tr) at (2,2.7) {$Z$};
    \node (tr) at (2,2) {$Z \otimes I \otimes I$};
    \node (br) at (2,1) {$Z \otimes I \otimes S_g$};
    \node (Br2) at (2,.1) {$Z \otimes S_f \otimes S_g$};
    \node (Br) at (2,-.5) {$Z \otimes S_{gf}$};
    \draw[->] (Tl) to node[above]{$f$} (T);
    \draw[->] (T) to node[above]{$g$} (Tr);
    \draw[->] (tl) to node[above]{$f \otimes I$} (t);
    \draw[->] (t) to node[above]{$g \otimes I$} (tr);
    \draw[->] (bl) to node[above]{$f \otimes S_f$} (b1);
    \draw[->] (b2) to node[above]{$g \otimes S_g$} (br);
    \draw[->] (Bl) to node[above]{$f \otimes S_{gf}$} (B);
    \draw[->] (B) to node[above]{$g \otimes S_{gf}$} (Br);
    \draw[->] (Tl) to node[right]{$\lambda^{-1}$} (tl);
    \draw[->] (T) to node[right]{$\lambda^{-1}$} (t);
    \draw[->] (Tr) to node[left]{$\lambda^{-1}$} (tr);
    \draw[->] (tl) to node[right]{$\id \otimes d_f$} (bl);
    \draw[>->] (br) to node[left]{$\id \otimes s_g$} (tr);
    \draw[>->] (b1) to node[left]{$\id \otimes s_f \otimes s_f$} (t);
    \draw[->] (t) to node[right]{$\id \otimes d_g$} (b2);
    \draw[->] (bl) to node[right]{$\id \otimes d_g$} (Bl2);
    \draw[->] (b1) to node[left]{$\id \otimes d_g$} (B2);
    \draw[>->] (B2) to node[right]{$\id \otimes s_f \otimes \id$} (b2);
    \draw[>->] (Br2) to node[left]{$\id \otimes s_f \otimes \id$} (br);
    \draw[-,double] (Bl2) to (Bl);
    \draw[-,double] (B2) to (B);
    \draw[-,double] (Br2) to (Br);
    \draw[->] (tl) to[out=-130,in=130,looseness=.6] node[right]{$\id \otimes d_{gf}$} (Bl);
    \draw[>->] (Br) to[out=60,in=-60,looseness=.6] node[left]{$\id \otimes s_{gf}$} (tr);
    \draw[->] (Bl) to[out=-60,in=-130,looseness=.8] node[below]{$(g \circ f) \otimes S_{gf}$} (Br);
  \end{pic}\]
  Hence $s_{g \circ f} = s_g \wedge s_f$. It follows that $\rest{gf} = e_{gf} \bullet X = e_g \bullet e_f \bullet X = e_g \bullet \rest{f}$. Therefore the uniqueness of (TR4) guarantees that the following diagram commutes:
  \[\begin{pic}[xscale=5,yscale=1.5]
    \node (tl) at (0,2) {$X \otimes S_{gf}$};
    \node (t) at (1,2) {$Z \otimes S_{gf}$};
    \node (tr) at (2,2) {$Z$};
    \node (bl) at (0,1) {$X \otimes S_f \otimes S_g$};
    \node (b) at (1,1) {$Z \otimes S_f \otimes S_g$};
    \node (br) at (2,1) {$Z \otimes S_g$};
    \node (l) at (.5,0) {$Y \otimes S_f \otimes S_g$};
    \node (r) at (1.5,0) {$Y \otimes S_g$};
    \draw[->] (tl) to node[above]{$\total{(g \circ f) \otimes \id}$} (t);
    \draw[>->] (t) to node[above]{$\id \otimes s_{gf}$} (tr);
    \draw[-,double] (tl) to (bl);
    \draw[-,double] (t) to (b);
    \draw[>->] (br) to node[right]{$\id \otimes s_g$} (tr);
    \draw[->] (bl) to node[left]{$\total{f \otimes \id} \otimes \id$} (l);
    \draw[->] (l) to node[right]{$\sigma \circ (\total{g \otimes \id} \otimes \id) \circ \sigma$} (b);
    \draw[>->] (b) to node[above]{$\id \otimes s_f \otimes \id$} (br);
    \draw[>->] (l) to node[below]{$\id \otimes s_f \otimes \id$} (r);
    \draw[->] (r) to node[right]{$\total{g \otimes \id}$} (br);
  \end{pic}\]
  Thus $F(g \circ f) = [s_{g \circ f},T(g \circ f)] = [s_g \wedge s_f, T(g) \circ (T(f) \otimes \id)] = F(g) \circ F(f)$.

  In the other direction, define $G \colon \S[\T[\cat{X}]] \to \cat{X}$ by $G\big[s,\;\begin{minipic}
        \node[morphism] (f) at (0,0) {$f$};
        \draw ([xshift=-1mm]f.south west) to +(0,-.125) node[right=-1mm,font=\tiny]{$A$};
        \draw ([xshift=1mm]f.south east) to +(0,-.125) node[right=-1mm,font=\tiny]{$S$};
        \draw (f.north) to +(0,.125) node[right=-1mm,font=\tiny]{$B$};
      \end{minipic}\!\big]=f \circ (\id \otimes d_f)$. This is a well-defined functor by Lemmas~\ref{lem:OtoISub} and~\ref{lem:ISubtoO}. Clearly on objects we have that $F(G(X))= X$ and $G(F(X))$. For morphisms, on the one hand we have $G(F(f)) = T(f) \circ (\id \otimes d_f) = f$ by construction. On the other hand,
  for restriction-total $f \colon X \otimes S \to Y$ we have $s_{\id \otimes d_f} = \id$, and so:
  \begin{align*}
    F(G\big[s,\;\begin{minipic}
        \node[morphism] (f) at (0,0) {$f$};
        \draw ([xshift=-1mm]f.south west) to +(0,-.125) node[right=-1mm,font=\tiny]{$A$};
        \draw ([xshift=1mm]f.south east) to +(0,-.125) node[right=-1mm,font=\tiny]{$S$};
        \draw (f.north) to +(0,.125) node[right=-1mm,font=\tiny]{$B$};
      \end{minipic}\!\big]) 
    & = F(f \circ (\id \otimes d_f)) \\
    & = F(f) \circ F(X \otimes d_f) \\
    & = [s_f, T(f)] \circ [\id[I], d_f] \\
    & = [s, T(f) \circ (\id \otimes d_f)] \\
    & = [s, T(f) \circ (\rest{f} \otimes d_f)] \\
    & = \big[s,\;\begin{minipic}
        \node[morphism] (f) at (0,0) {$f$};
        \draw ([xshift=-1mm]f.south west) to +(0,-.125) node[right=-1mm,font=\tiny]{$A$};
        \draw ([xshift=1mm]f.south east) to +(0,-.125) node[right=-1mm,font=\tiny]{$S$};
        \draw (f.north) to +(0,.125) node[right=-1mm,font=\tiny]{$B$};
      \end{minipic}\!\big]
  \end{align*}
  Thus $F$ and $G$ are inverses.

  It is clear that $F$ and $G$ are strict monoidal functors.
  It follows directly that $F$ and $G$ preserve subunits, and indeed $G\big[\id[I],\,\begin{minipic}
      \draw (.3,-.2) node[dot]{} to (.3,-.4) node[right=-1mm,font=\tiny]{$S$};
    \end{minipic}\!\big]=s$, and $F(s) =\big[\id[I],\,\begin{minipic}
      \draw (.3,-.2) node[dot]{} to (.3,-.4) node[right=-1mm,font=\tiny]{$S$};
    \end{minipic}\!\big]$.

  Finally, $T(\rest{f}) = T(e_f \bullet \id) = T(e_f) \bullet \id = \id \otimes s_f$, therefore: 
  \[F(\rest{f}) = [s_f, T(\rest{f})] = [s_f, \id \otimes s_f] = \rest{[s_f, T(f)]} = \rest{F(f)}\]
  so $F$ preserves restriction. Similarly, 
  \begin{align*}
    G(\rest{\big[s,\;\begin{minipic}
        \node[morphism] (f) at (0,0) {$f$};
        \draw ([xshift=-1mm]f.south west) to +(0,-.125) node[right=-1mm,font=\tiny]{$A$};
        \draw ([xshift=1mm]f.south east) to +(0,-.125) node[right=-1mm,font=\tiny]{$S$};
        \draw (f.north) to +(0,.125) node[right=-1mm,font=\tiny]{$B$};
      \end{minipic}\!\big]}) 
    & = G \big[s,\,\begin{minipic}
      \draw (0,0) to (0,-.5) node[right=-1mm,font=\tiny]{$A$};
      \draw (.3,-.2) node[dot]{} to (.3,-.5) node[right=-1mm,font=\tiny]{$S$};
    \end{minipic}\!\big] = (\id \otimes s) \circ (\id \otimes d_s) = \id \otimes e_f 
    \\
    & = \id \otimes \rest{d_f} = \rest{f \circ (\id \otimes d_f)} = \rest{G\Big[s,\;\begin{minipic}
        \node[morphism] (f) at (0,0) {$f$};
        \draw ([xshift=-1mm]f.south west) to +(0,-.125) node[right=-1mm,font=\tiny]{$A$};
        \draw ([xshift=1mm]f.south east) to +(0,-.125) node[right=-1mm,font=\tiny]{$S$};
        \draw (f.north) to +(0,.125) node[right=-1mm,font=\tiny]{$B$};
      \end{minipic}\!\big]}
  \end{align*}
  so $G$ preserves restriction too. So we conclude that it is an isomorphism of tensor-restriction categories.
\end{proof}

This main result easily lifts to functors. Recall that a \emph{morphism} of firm monoidal categories~\cite[10.1]{enriquemolinerheunentull:tensortopology} is a (strong) monoidal functor that sends subunits to subunits. They form a category $\cat{FirmCat}$.

\begin{definition}
  A \emph{morphism of tensor-restriction categories} is a functor $F \colon \cat{X} \to \cat{Y}$ that is (strong) monoidal, sends subunits to subunits, and satisfies $F(\rest{f})=\rest{F(f)}$ (i.e. $F$ is a restriction functor \cite[Section 2.2.1]{cockettlack:restrictioncategories}).
  Tensor-restriction categories and their morphisms form a category $\cat{TensRestCat}$.
\end{definition}

\begin{theorem}
  There is an equivalence of categories $\cat{FirmCat} \simeq \cat{TensRestCat}$.
\end{theorem}
\begin{proof}
  All that remains to be verified is that $\S$ and $\T$ are functors. But this is easy: define $\S[F]$ by $A \mapsto F(A)$ on objects and by $[s,f] \mapsto [F(s), F(f)]$ on morphisms, and define $\T[G]$ by $A \mapsto G(X)$ on objects and as $G(f)=f$ on morphisms.
\end{proof}

This lets us characterise the fixed points of the $\S$-construction.
Recall that a firm category is \emph{simple} when it has no nontrivial subunits~\cite[Definition~5.3]{enriquemolinerheunentull:tensortopology}.

\begin{corollary}\label{cor:simple}
  A firm monoidal category $\cat{C}$ has $\S[\cat{C}] \simeq \cat{C}$ if and only if it is simple.
\end{corollary}
\begin{proof}
  If $\ISub(\cat{C})=\{1\}$, every map in $\S[\cat{C}]$ is total, and so $\S[\cat{C}]=\T[\S[\cat{C}]]\simeq\cat{C}$.
  Conversely, if $S[\cat{C}] \simeq \cat{C} \simeq \T[\S[\cat{C}]]$, then every map in $\S[\cat{C}]$ is total, and so $\ISub(\cat{C})=\{1\}$.
\end{proof}

A particular example of a simple firm monoidal category $\cat{C}$ is when $\cat{C}$ is a groupoid. Thus as a consequence of Corollary~\ref{cor:inversecat}, if $\cat{C}$ is a groupoid, or equivalently if $\S[\cat{C}]$ is an inverse category, then $\S[\cat{C}] \simeq \cat{C}$.  


We conclude this section with examples pointing out that not every firm monoidal category is a tensor-restriction category. 

\begin{example} 
  Every category $\cat{C}$ is trivially a restriction category by setting $\rest{f} = 1$, so every map is total and $\T[\cat{C}] = \cat{C}$. Thus every firm monoidal category is trivially a firm monoidal restriction category. Suppose a tensor-restriction category $\cat{X}$ had a trivial restriction, so $\T[\cat{X}] = \cat{X}$. Theorem~\ref{thm:main} then shows that $\cat{X} \cong \S[\T[\cat{X}]] = \S[\cat{X}]$. But Corollary~\ref{cor:simple} then implies that $\cat{X}$ is simple. Thus every firm monoidal category that is not simple (i.e.\ has a nontrivial subunit), regarded as a firm monoidal restriction category with the trivial restriction, is not a tensor-restriction category. 
\end{example}

\begin{example} 
  There are also many examples of firm monoidal restriction categories with nontrivial restrictions that are not tensor-restriction categories. Consider $\cat{Par}$ and recall that $\T[\cat{Par}] = \cat{Set}$. If it was a tensor-restriction category, Theorem~\ref{thm:main} would give $\cat{Par} \simeq \S[\T[\cat{Set}]]$, but this is not the case. Indeed, it is clear that $\cat{Par}(X,Y) \not\simeq \cat{Set}(X,Y) + 1 \sim \S[\cat{Set}](X,Y)$ as in Example~\ref{ex:sset}). Therefore $\cat{Par} \not\simeq \S[\T[\cat{Set}]]$. Thus $\cat{Par}$ is not a tensor-restriction category. 
\end{example}

\section{Alternative axiomatisations}\label{sec:alternatives}

The axiomatisation of tensor-restriction categories of Definition~\ref{def:tensorrestrictioncategory} has several features. First, of course, it works, in the sense that Theorem~\ref{thm:main} holds. Second, it is elementary, in the sense that it is phrased entirely in terms of basic notions from restriction category theory and tensor topology. Third, it is intuitive in that it conveys the structure of the $\S$-construction.
Nevertheless, there is room for alternative axiomatisations, which we discuss in this section. In this sense, this section is not necessarily essential to this article, but it does open the door to future investigation. We discuss two aspects:
\begin{itemize}
  \item It is not clear that axioms (TR1)--(TR7) are independent. In fact, there does appear to be some redundancy, which we will point out. We will argue that the ``property-based'' axiomatisation of Definition~\ref{def:tensorrestrictioncategory} may be replaced by a ``structure-based'' one where the tensor-restriction points $d$ are given. 
  \item The ``uniformity'' axiom (TR3) means that $\S[\cat{C}]$ is not just a restriction category, but its opposite is too. This may make for a more efficient axiomatisation. We will make a start towards such an alternative axiomatisation by defining \emph{bi-restriction categories} and observing that $\S[\cat{C}]$ is one.
\end{itemize}

Starting with the first goal, of taking the maps $d$ as structure rather than emergent properties, we first make precise the type of these maps.

\begin{definition} In a firm restriction category $\cat{X}$, a  \emph{restriction-subunit point} of subunit a $s \colon S \rightarrowtail I$  is a morphism $d \colon I \to S$ satisfying $\rest{d} = s \circ d$ .
\end{definition}
The next three lemmas show that these restriction-subunit points have some of the properties of axioms (TR1)--(TR7) automatically.

\begin{lemma}
In a firm restriction category, if $d\colon I \to S$ and $d'\colon I \to T$ are restriction-subunit points, then $d \wedge d' \colon: I \to S \otimes T$, defined as the composite $d \wedge d^\prime = (d \otimes d^\prime) \circ \lambda^{-1}$ is a restriction-subunit point of $S \otimes T$. 
\end{lemma}
\begin{proof}
The computation
\begin{align*}
 \rest{d \wedge d^\prime} &= \rest{(d \otimes d^\prime) \circ \lambda^{-1}} \\
 &= \lambda \circ (\rest{d} \otimes \rest{d^\prime}) \circ \lambda^{-1} \\
 &= \lambda \circ (s \otimes t) \circ (d \otimes d^\prime) \circ \lambda^{-1} \\
 &= (s \wedge t) \circ (d \wedge d^\prime)
\end{align*}
shows that $(d \wedge d')$ is a restriction-subunit point of $S \otimes T$. 
\end{proof}

\begin{lemma}
  In a firm restriction category, any restriction-subunit point $d \colon I \to S$ is a restriction isomorphism with restriction inverse $d^\circ = \rest{d} \circ s$.
\end{lemma}
\begin{proof} 
  First note that in a monoidal restriction category,  $\rest{a \bullet f} = \rest{a} \bullet \rest{f}$ for any scalar $a\colon I \to I$ and any map $f \colon X \to Y$. Since $s$ is total,
  $
  \rest{d^\circ} = \rest{\rest{d} \bullet s} = \rest{\rest{d}} \bullet \rest{s} = \rest{d} \bullet I = \rest{d}
  $
  and therefore
\begin{align*}
    d \circ d^\circ = d \circ \rest{d} \circ s = d \circ s = \rest{d} = \rest{d^\circ}\text.
\end{align*}
  Similarly
  $
    d^\circ \circ d = \rest{d} \circ s \circ d = \rest{d} \circ \rest{d} = \rest{d}
  $, 
  so $d$ is a restriction isomorphism. 
\end{proof}

\begin{lemma} 
  Let $d \colon I \to S$ be a restriction-subunit point in a firm restriction category.
  Suppose that it satisfies the left-lifting property against subunits:
  if $t \circ f = g \circ d_s$ for any subunit $T$ and maps $f \colon I \to T$ and $g \colon S \to I$, then $f = m \circ d$ and $g = t \circ m$ for some $m$:
     \[ \begin{pic}[xscale=2,yscale=1.5]
        \node (tl) at (0,1) {$I$};
        \node (tr) at (1,1) {$S$};
        \node (bl) at (0,0) {$T$};
        \node (br) at (1,0) {$I$};
        \draw[->] (tl) to node[above]{$d$} (tr);
        \draw[->] (tr) to node[right]{$g$} (br);
        \draw[->] (tl) to node[left]{$f$} (bl);
        \draw[>->] (bl) to node[below]{$t$} (br);
        \draw[->,dashed] (tr) to node[right=2mm]{$m$} (bl);
       \end{pic}\]
  Then $d^\circ \colon S \to I$ is tensor-total, and $d \colon I \to S$ is a tensor-restriction point. 
\end{lemma}
\begin{proof}
  Suppose $d^\circ$ factors through a subunit $T$ as follows:
      \[\begin{pic}[xscale=2]
        \node (l) at (0,0) {$S$};
        \node (r) at (2,0) {$I$};
        \node (t) at (1,1) {$T$};
        \draw[->] (t) to node[above]{$t$} (r);
        \draw[->,halo] (l) to node[below]{$d^\circ$} (r);
        \draw[->] (l) to node[above]{$f$} (t);
      \end{pic}\]
      We need to show that $S \otimes T \simeq S$. To do so, we will show that $s \leq t$. Observe:
    \[ s \circ d = \rest{d} = d^\circ \circ d = t \circ f \circ d \] 
      The left lifting property gives $m\colon S \to T$ with: 
      \[  \begin{pic}[xscale=1.5,yscale=1.25]
        \node (tl) at (0,2) {$I$};
        \node (tr) at (2,2) {$S$};
        \node (bl) at (0,0) {$T$};
        \node (br) at (2,0) {$I$};
        \node (ml) at (0,1) {$S$} ;
        \draw[->] (tl) to node[above]{$d$} (tr);
        \draw[>->] (tr) to node[right]{$s$} (br);
        \draw[->] (tl) to node[left]{$d$} (ml);
          \draw[->] (ml) to node[left]{$f$} (bl);
        \draw[>->] (bl) to node[below]{$t$} (br);
        \draw[->,dashed] (tr) to node[right=2mm]{$m$} (bl);
       \end{pic}
      \]
  In particular, $t \circ m = s$, and therefore $s \leq t$. We conclude that $S \otimes T \simeq S$, and so $d^\circ$ is tensor-total. Since $d^\circ$ is also a restriction isomorphism, it follows by definition that $d$ is a tensor-restriction point. 
\end{proof}

The previous three lemmas seem to exhibit some redundancy in the properties that (TR1)--(TR7) ask of the maps $d$. However, to progress in obtaining an equivalent axiomatisation, we need to demand a universal property of the restriction-subunit points $d$.
Say that a subunit $S$ has a \emph{maximal} restriction-subunit point if there is a restriction subunit point $d_s \colon I \to S$ that is initial amongst restriction-subunit points of $S$. More precisely, every restriction-subunit point $d \colon I \to S$ satisfies $d \leq d_s$ in the restriction category sense, that is $d = d_s \circ \rest{d}$.

\begin{lemma}
  In a firm restriction category, if subunits $S$ and $T$ have maximal restriction-subunit point $d_s \colon I \to S$ and $d_t\colon I \to T$, then $d_s \wedge d_t = (d_s \otimes d_t) \circ \lambda^{-1} \colon I \to S \otimes T$ is a maximal restriction-subunit point of $S \otimes T$.
\end{lemma}
\begin{proof}
  Suppose $d \colon I \to S \otimes T$ is a restriction subunit point of $S \otimes T$, that is, $\rest{d}= (s \wedge t) \circ d$. We need to show that $d \leq d_s \wedge d_t$. First observe 
\begin{align*}
    \rest{\lambda \circ (s \otimes T) \circ d} = \rest{d} = (s \wedge t) \circ d = \lambda \circ (s \otimes t) \circ d = t \circ \lambda \circ (s \otimes T) \circ d\text,
\end{align*}
where the first equality holds by totality of $\lambda \circ (s \otimes T)$. Therefore, $\lambda \circ (s \otimes T) \circ d$ is a restriction-subunit point of $T$ and so $\lambda \circ (s \otimes T) \circ d \leq d_t$:
\[ \lambda \circ (s \otimes T) \circ d = d_t \circ \rest{d} \]
Similarly $\rho \circ (S \otimes t) \circ d$ is a restriction-subunit point of $S$ and so $\rho \circ (S \otimes t) \circ d \leq d_s$:
\[ \rho \circ (S \otimes t) \circ d = d_s \circ \rest{d}  \]
Now we can show that $(d_s \wedge d_t) \circ \rest{d} =d$:
\begin{align*}
    (d_s \wedge d_t) \circ \rest{d} &=    (d_s \wedge d_t) \circ \rest{d} \circ \rest{d} \\
    &= (d_s \otimes d_t) \circ \lambda^{-1} \circ \rest{d} \circ \rest{d} \\
    &= (d_s \otimes d_t) \circ (\rest{d} \otimes \rest{d}) \circ \lambda^{-1} \\
    &= \left( \left(\lambda \circ (s \otimes T) \circ d \right) \otimes \left(\rho \circ (S \otimes t) \circ d \right) \right) \circ  \lambda^{-1} \\
    &= \left( \left(\lambda \circ (s \otimes t) \circ d \right) \otimes d \right) \circ \lambda^{-1} \\
    &= \left( \left( (s \wedge t) \circ d \right) \otimes d \right) \circ \lambda^{-1} \\
    &= \left( \rest{d} \otimes d \right) \circ \lambda^{-1} \\
    &= d \circ \rest{d} \\
    &= d
\end{align*}
Therefore $d \leq d_s \wedge d_t$. 
\end{proof}

As an intermezzo, we now discuss how to use the tensor-restriction points $\smash{\begin{minipic} \node[dot] (d) at (0,.1){}; \draw (d) to +(0,.2) node[right=-1mm]{$S$}; \end{minipic}}$in a graphical normal form of maps in tensor-restriction categories. This strongly resembles a monoidal version of~\cite{nester:cartesianrestrictiondiagrams} and is a first step towards a general calculus handling restriction graphically. The factorisation system of Proposition~\ref{prop:factorisation} lets us decompose any morphism $f$ in a tensor-restriction category as follows:
\[
  \begin{pic}
    \node[morphism,width=5mm] (f) at (0,0) {$f$};
    \draw (f.north) to +(0,.5) node[above]{$B$};
    \draw (f.south) to +(0,-.5) node[below]{$A$};
  \end{pic}
  \;\;=\;\;
  \begin{pic}
    \draw[dashed,gray] (-.6,-.3) to (3.5,-.3);
    \node[morphism,width=8mm] (f) at (0,0) {$T(f)$};
    \node[dot] (d) at ([yshift=-4mm]f.south east) {};
    \draw (d) to node[below right]{$S_f$} (f.south east);
    \draw (f.south west) to +(0,-.5) node[below]{$A$};
    \draw (f.north) to +(0,.5) node[above]{$B$};
    \node[gray,anchor=west] at (1,.2) {total part};
    \node[gray,anchor=west] at (1,-.7) {restriction part};
  \end{pic}
\]
In $\S[\cat{C}]$, this says
$\big[s,\,\begin{minipic} \node[morphism] (f) at (0,0) {$f$}; \draw(f.north) to +(0,.125); \draw (f.south west) to +(0,-.125); \draw (f.south east) to +(0,-.125); \end{minipic}\,\big]
=
\big[1,\,\begin{minipic} \node[morphism] (f) at (0,0) {$f$}; \draw(f.north) to +(0,.125); \draw (f.south west) to +(0,-.125); \draw (f.south east) to +(0,-.125); \end{minipic}\,\big]
\circ
\Big(
\big[1,\begin{minipic} \draw (0,.1) to (0,-.2) node[right=-1mm]{$A$}; \end{minipic}\!\big]
\otimes
\big[s,\begin{minipic} \draw (0,.1) to (0,-.2) node[right=-1mm]{$S$}; \end{minipic}\!\big]
\Big)
$.
Restriction and composition become:
\[
  \begin{pic}
    \node[morphism,width=5mm] (f) at (0,0) {$\rest{f}$};
    \draw (f.north) to +(0,.5) node[above]{$A$};
    \draw (f.south) to +(0,-.5) node[below]{$A$};
  \end{pic}
  = 
  \begin{pic}
    \draw[dashed,gray] (-.2,-.3) to (.5,-.3);
    \draw (0,-1) node[below]{$A$} to (0,.4) node[above]{$A$};
    \draw (.3,-.7) node[dot]{} to node[right]{$S_f$} (.3,.1) node[dot]{};
  \end{pic}
  \qquad\qquad
  \begin{pic}
    \node[morphism,width=5mm] (f) at (0,-.4) {$f$};
    \node[morphism,width=5mm] (g) at (0,.4) {$g$};
    \draw (g.north) to +(0,.4) node[above]{$C$};
    \draw (f.north) to node[right]{$B$} (g.south);
    \draw (f.south) to +(0,-.4) node[below]{$A$};
  \end{pic}
  \;=\;
  \begin{pic}
    \draw[dashed,gray] (-.8,-.5) to (1,-.5);
    \node[morphism,width=8mm] (f) at (0,0) {$T(g \circ f)$};
    \node[dot] (d) at ([yshift=-6mm]f.south east) {};
    \draw (d) node[below]{$S_g$} to (f.south east);
    \node[dot] (d2) at ([yshift=-6mm]f.south) {};
    \draw (d2) node[below]{$S_f$} to (f.south);
    \draw (f.south west) to +(0,-1.1) node[below]{$A$};
    \draw (f.north) to +(0,.5) node[above]{$C$};
  \end{pic}
\]

Finally, we change to the second goal of this section, by observing that the $\S[-]$-construction also induces a corestriction category~\cite{cockett:range, cockettlack:restrictioncategories}.

\begin{definition} \cite[Example 2.1.3.12]{cockettlack:restrictioncategories} A \emph{corestriction category} is a category $\cat{X}$ equipped with a choice of endomorphism $\range{f} \colon B \to B$ for each morphism $f \colon A \to B$ satisfying:
\begin{align}  
   \range{f} \circ f & = f \tag{CR1} \\
  \range{f} \circ \range{g} & = \range{g} \circ \range{f} &&\text{if $\cod(f)=\cod(g)$} \tag{CR2} \\
  \range{\range{g} \circ f} & = \range{g} \circ \range{f} &&\text{if $\cod(f)=\cod(g)$} \tag{CR3} \\
 g \circ \range{f} & = \range{g \circ f} \circ g &&\text{if $\dom(g)=\cod(f)$} \tag{CR4}
\end{align}
We call $\range{f}$ the \emph{corestriction} of $f$.
\end{definition}

\begin{example} Let $\cat{stabLat}$ be the category whose objects are semilattices and whose morphisms are stable homomorphisms, that is, morphisms between semilattices that preserve binary meets but not necessarily the top element. This is a corestriction category, where for a stable homomorphism $f \colon L \to L'$, the corestriction is the stable homomorphism $\range{f} \colon L' \to L'$ given by $\range{f}(x) = x \wedge f(1)$.  
\end{example}

The corestriction structure of $\S[\cat{C}]$ is defined similarly to its restriction structure.

\begin{proposition}\label{prop:SCcorestriction}
  If $\cat{C}$ is a firm monoidal category, $\S[\cat{C}]$ is a corestriction category with:
  \[
    \range{ \big[s,\;\begin{minipic}
        \node[morphism] (f) at (0,0) {$f$};
        \draw ([xshift=-1mm]f.south west) to +(0,-.125) node[right=-1mm,font=\tiny]{$A$};
        \draw ([xshift=1mm]f.south east) to +(0,-.125) node[right=-1mm,font=\tiny]{$S$};
        \draw (f.north) to +(0,.125) node[right=-1mm,font=\tiny]{$B$};
      \end{minipic}\!\big]} = 
      \big[s,\,\begin{minipic}
      \draw (0,0) to (0,-.5) node[right=-1mm,font=\tiny]{$B$};
      \draw (.3,-.2) node[dot]{} to (.3,-.5) node[right=-1mm,font=\tiny]{$S$};
    \end{minipic}\!\big]
  \]
\end{proposition}
\begin{proof} 
Completely analogous to Proposition \ref{prop:SCrestriction}. 
\end{proof}

It turns out that the corestriction structure of the previous proposition also makes $\S[\cat{C}]$ into a range category.

\begin{definition}\cite[Definition 2.12]{cockettlack:restrictioncategories}. A \emph{range category} is a restriction category $\cat{X}$ which is additionally equipped with a choice of endomorphism $\range{f} \colon B \to B$ for each morphism $f \colon A \to B$ satisfying: 
\begin{align}  
  \rest{\range{f}} & = \range{f} \tag{RR1} \\
  \range{f} \circ f & = f \tag{RR2} \\
 \range{\rest{g} \circ f}& = \rest{g} \circ \range{f} &&\text{if $\cod(f)=\dom(g)$} \tag{RR3} \\
  \range{g \circ \range{f}}& = \range{g \circ f} &&\text{if $\cod(f)=\dom(g)$} \tag{RR4}
\end{align}
We call $\range{f}$ the range of $f$. 
\end{definition}

Note that not every range category is automatically a corestriction category. Indeed, being a corestriction category does not require restriction structure, whereas being a range category does. Here is the paradigmatic example.

\begin{example} $\cat{Par}$ is a range category where the range of a for a partial function $f \colon X \to Y$ is defined as follows: 
  \begin{align*}
\range{f}(y) = \begin{cases} 
y  & \text{ if $f(x)=y$ for some $x \in X$}  \\
  \text{undefined} & \text{ otherwise } 
    \end{cases}
  \end{align*}
  However, this does not make $\cat{Par}$ a corestriction category since (CR4) fails. 
\end{example}

If a category is both a restriction category and a corestriction category, and if the restriction and corestriction operators are compatible, then in fact the corestriction operator is also a range category. We introduce a new notion of \emph{birestriction category} (which is a stronger version of a bisupport category \cite{cockettlack:restrictioncategories}).

\begin{definition} A \emph{birestriction category} is a category equipped with both a restriction operator $\rest{\phantom{f}}$ and a corestriction operator $\range{\phantom{f}}$ satisfying: 
\begin{align}  
   \range{\rest{f}} & = \rest{f} \tag{BR1} \\
  \rest{\range{f}} & = \range{f} \tag{BR2}
\end{align}
\end{definition}

\begin{example}
  If $\cat{X}$ is both a restriction category and a dagger category~\cite{heunenkarvonen:daggermonad}, and restriction idempotents are self-adjoint, that is, $\rest{f}^\dag = \rest{f}$, then $\cat{X}$ is automatically a birestriction category with $\range{f} = \rest{f^\dag}$.
\end{example}

\begin{lemma} 
  In a birestriction category, the corestriction operator is a range operator for the restriction operator. 
\end{lemma}
\begin{proof} Let $\rest{\phantom{f}}$ be the restriction operator and $\range{\phantom{f}}$ be the corestriction operator. We need to show that $\range{\phantom{f}}$ satisfies (RR1)--(RR4). The first two are immediate since (RR1) is simply (BR2) and (RR2) is (CR1). For (RR3), we use (CR3) and (BR1):
\[ \range{\rest{g} \circ f} = \range{\range{\rest{g}} \circ f} = \range{\rest{g}} \circ \range{f} = \rest{g} \circ \range{f} \]
Finally, (RR4) is the dual of \cite[Lemma 2.1.(iii)]{cockettlack:restrictioncategories}. So $\range{\phantom{f}}$ is a range operator. 
\end{proof}

The point of introducing birestriction categories is that the $\S$-construction induces one, which may lead to an alternative axiomatisation.

\begin{proposition}  If $\cat{C}$ is a firm monoidal category, then $\S[\cat{C}]$ is a birestriction category.
\end{proposition}
\begin{proof} 
  It is straightforward to see that both (BR1) and (BR2) hold in $\S[\cat{C}]$. 
\end{proof}

\bibliographystyle{plain}
\bibliography{bibliography}

\end{document}